\documentclass[12pt,leqno]{article}

\usepackage{graphicx}

\usepackage{palatino}
\usepackage[mathcal]{euler}

\usepackage{amsmath,amsthm,amsfonts,amssymb,latexsym,amscd,enumerate}

\usepackage{xy}
\xyoption{all}

\theoremstyle{plain}
    \newtheorem{theorem}{Theorem}[section]
    \newtheorem{lemma}[theorem]{Lemma}
    \newtheorem{corollary}[theorem]{Corollary}
    \newtheorem{proposition}[theorem]{Proposition}

 \theoremstyle{definition}
    \newtheorem{definition}[theorem]{Definition}
    
    \newtheorem{example}[theorem]{Example}

    \newtheorem{remark}[theorem]{Remark}

\theoremstyle{remark}

\numberwithin{equation}{section}

\DeclareMathOperator{\tr}{tr}
\DeclareMathOperator{\Ad}{Ad}
\DeclareMathOperator{\ad}{ad}
\DeclareMathOperator{\ind}{index}
\DeclareMathOperator{\grad}{grad}
\DeclareMathOperator{\End}{End}
\DeclareMathOperator{\Hom}{Hom}

\DeclareMathOperator{\topp}{top}
\DeclareMathOperator{\reg}{reg}
\DeclareMathOperator{\mat}{mat}
\DeclareMathOperator{\diag}{diag}

\DeclareMathOperator{\Spin}{Spin}
\DeclareMathOperator{\OO}{O}
\DeclareMathOperator{\SO}{SO}

\DeclareMathOperator{\Cl}{Cl}

\begin{document}


\newcommand{\myemph}{\emph}

\newcommand{\Spinc}{\Spin^c}

    \newcommand{\R}{\mathbb{R}}
    \newcommand{\C}{\mathbb{C}} 
    \newcommand{\N}{\mathbb{N}}
    \newcommand{\Z}{\mathbb{Z}} 
    \newcommand{\Q}{\mathbb{Q}}
     \newcommand{\bH}{\mathbb{H}}

\newcommand{\kk}{\mathfrak{k}}  
\newcommand{\kt}{\mathfrak{t}}
\newcommand{\kg}{\mathfrak{g}}  
\newcommand{\kh}{\mathfrak{h}} 
\newcommand{\ka}{\mathfrak{a}} 

\newcommand{\cE}{\mathcal{E}}
\newcommand{\cS}{\mathcal{S}}
\newcommand{\cL}{\mathcal{L}}
\newcommand{\cH}{\mathcal{H}}
\newcommand{\cO}{\mathcal{O}}
\newcommand{\cP}{\mathcal{P}}
\newcommand{\cD}{\mathcal{D}}
\newcommand{\cF}{\mathcal{F}}
\newcommand{\cX}{\mathcal{X}}

\newcommand{\ddt}{\left. \frac{d}{dt}\right|_{t=0}}
\newcommand{\Sj}{ \sum_{j = 1}^{\dim M}}

\newcommand{\indK}{\ind_K^{L^2}}
\newcommand{\ii}{\sqrt{-1}}
\newcommand{\nabLC}{\nabla^{TM}}
\newcommand{\QSpinc}{Q^{\Spinc}}
\newcommand{\vf}{v}
\newcommand{\rhoK}{\rho}
\newcommand{\specialin}{\hspace{-1mm} \in \hspace{1mm} }

\title{Equivariant indices of $\Spinc$-Dirac operators for proper moment maps}  
\date{\today}
\author{Peter Hochs\footnote{
University of Adelaide, \texttt{peter.hochs@adelaide.edu.au}}
\hspace{0mm}
 and Yanli Song\footnote{University of Toronto, \texttt{songyanl@math.utoronto.ca}} } 

\maketitle

\begin{abstract}
We define an equivariant index of $\Spinc$-Dirac operators on possibly noncompact manifolds, acted on by compact, connected Lie groups. 
The main result in this paper is that the index
decomposes into irreducible representations according to the quantisation commutes with reduction principle. 
\end{abstract}

\tableofcontents


\section{Introduction}

Let $M$ be a possibly noncompact, connected, even-dimensional manifold, on which a compact, connected Lie group $K$ acts. Suppose $M$ has a $K$-equivariant $\Spinc$-structure. An equivariant connection on the determinant line bundle naturally induces
a map from $M$ to the Lie algebra of $K$. This is called a moment map, and is defined via a formula due to Kostant. Assuming this moment map to be proper, we define an equivariant index of $\Spinc$-Dirac operators in this setting. This generalises an index of Dirac-type operators defined by Braverman \cite{Braverman02}, in the $\Spinc$-case.
We show that the index has a multiplicativity property, which allows us to prove that it satisfies the quantisation commutes with reduction principle.
This result yields a geometric way to
decompose the index into irreducible representations. It generalises the main result by Paradan and Vergne in \cite{Paradan14, Paradan14CRAS} from compact to noncompact manifolds. At the same time, it is an analogue in the more general $\Spinc$-case of the
 results  for symplectic manifolds by Ma and Zhang in \cite{Zhang14}, and by Paradan in \cite{Paradan11}.

\subsection*{Quantisation and reduction}

The \myemph{quantisation commutes with reduction} principle goes back to Guillemin and Sternberg's 1982 paper \cite{Guillemin82}. They made this principle rigorous, and proved it, for compact Lie groups acting on compact K\"ahler manifolds. Their conjecture that quantisation commutes with reduction for general compact symplectic manifolds inspired a large and impressive body of work, which culminated in the proofs in the late 1990s by Meinrenken \cite{Meinrenken98}  and Meinrenken and Sjamaar  \cite{Meinrenken99}. The richness of this conjecture was illustrated by completely different proofs by Tian and Zhang \cite{Zhang98} and Paradan \cite{Paradan01}.

For a compact Lie group $K$ acting on a compact symplectic manifold $(M, \omega)$, the definition of geometric quantisation, attributed to Bott, is
\begin{equation} \label{eq def quant intro}
Q_K(M, \omega) := \ind_K(D^L),
\end{equation}
the equivariant index of a ($\Spinc$- or Dolbeault-) Dirac operator $D^L$ on $M$, coupled to a line bundle $L$ with first Chern class $[\omega]$. Reduction involves a \myemph{moment map} $\mu: M \to \kk^*$, such that for all $X \in \kk$,
\begin{equation} \label{eq mom map intro}
2\ii \mu_X = \cL_X - \nabla_{X^M},
\end{equation}
where $\mu_X \in C^{\infty}(M)$ is the pairing of $\mu$ and $X$, $\cL_X$ denotes the Lie derivative of sections of $L$, $\nabla$ is a connection on $L$ with curvature $-2 \ii \omega$, and $X^M$ is the vector field on $M$ induced by $X$. In the symplectic setting, the action by $K$ on $(M, \omega)$ represents a  symmetry of a physical system, and a moment map is the associated conserved quantity.
The \myemph{reduced space} at a value $\xi \in \kk^*$ of $\mu$ is
\[
M_{\xi} := \mu^{-1}(K\cdot \xi)/K,
\] 
and is again symplectic (if smooth) by Marsden and Weinstein's result \cite{Marsden74}. 

Quantisation commutes with reduction is the decomposition
\begin{equation} \label{eq [Q,R]=0 intro}
Q_K(M, \omega) = \bigoplus_{\lambda} Q(M_{\lambda}) \pi_{\lambda + \rho},
\end{equation}
where $\lambda$ runs over the dominant integral weights of $K$ (with respect to a maximal torus and positive root system), and $\pi_{\lambda + \rho}$ is the irreducible representation of $K$ with infinitesimal character $\lambda + \rho$, i.e.\ with highest weight $\lambda$. Here $\rho$ is half the sum of a choice of positive roots. (In the symplectic setting one usually parametrises irredicible representations by their highest weights, but in the $\Spinc$-setting considered in this paper, using infinitesimal characters is more natural.) In this way, one obtains a geometric formula for the multiplicity of any irreducible representation in the index \eqref{eq def quant intro}.

\subsection*{Generalisations}

After the equality \eqref{eq [Q,R]=0 intro} was proved for compact $M$ and $K$, the natural question arose what generalisations are possible and useful. Results have been obtained in several directions.
\subsubsection*{Cocompact actions}
Landsman \cite{Landsman05} proposed a definition of quantisation and reduction in terms of $K$-theory of $C^*$-algebras, for possibly noncompact groups $G$ acting on possibly noncompact symplectic manifolds $(M, \omega)$, as long as the action is cocompact, i.e.\ $M/G$ is compact. He conjectured that in this setting, quantisation commutes with reduction at the trivial representation, i.e.\ for $\lambda = 0$ in \eqref{eq [Q,R]=0 intro}. Results were obtained in \cite{Hochs09, Hochs15, Landsman08}. An asymptotic version of Landsman's conjecture was proved in \cite{Mathai10}, which was generalised to cases where $M/G$ may be noncompact in \cite{Mathai13}.
\subsubsection*{Compact groups, noncompact manifolds}
Vergne \cite{Vergne06} generalised the definition \eqref{eq def quant intro} of geometric quantisation to {compact  groups} and possibly {noncompact manifolds}, and conjectured that the appropriate generalisation of \eqref{eq [Q,R]=0 intro} still holds. Ma and Zhang \cite{Zhang14} generalised  \eqref{eq def quant intro} in a different way, assuming the moment map to be proper, which contains Vergne's definition as a special case. They gave an analytic proof that quantisation commutes with reduction for their definition. Paradan \cite{Paradan11} later used very different, topological, techniques to also generalise Vergne's definition and prove that quantisation commutes with reduction.
\subsubsection*{Compact $\Spinc$-manifolds}
It was noted in \cite{Cannas00} that {$\Spinc$-manifolds} provide the most general framework for studying geometric quantisation, and it was shown that quantisation commutes with reduction for circle actions on compact $\Spinc$-manifolds.
Recently, Paradan and Vergne \cite{Paradan15, Paradan14,  Paradan14CRAS} generalised this to actions by arbitrary compact, connected Lie groups. 
It is fascinating that quantisation commutes with reduction in this generality, so that it is a property of indices of $\Spinc$-Dirac operators rather than just of geometric quantisation.

 In the general $\Spinc$-setting, the term `quantisation' becomes restrictive, since Paradan and Vergne's result applies to all $\Spinc$-Dirac operators on compact manifolds, not just those used to define geometric quantisation. However, interpreting the index of such an operator as a quantisation makes it natural to generalise the quantisation commutes with reduction principle. The $\Spinc$-version of this principle has a great scope for applications, and for example implies Atiyah and Hirzebruch's vanishing theorem \cite{Atiyah70b} for $\Spin$-manifolds. The  results on Landsman's conjecture mentioned above were generalised to $\Spinc$-manifolds in \cite{Mathai14}.

\subsection*{The main result}

The results in \cite{Zhang14, Paradan11} (on compact groups and possibly noncompact symplectic manifolds) are based on two different generalisations of \eqref{eq def quant intro}, and analytic one and a topological one, which give the same result. Braverman \cite{Braverman02} defined a third equivariant index for compact groups acting on possibly noncompact manifolds. This is based on a deformation
\[
D_{\mu} := D - \ii c(\vf^{\mu})
\]
of a Dirac-type operator $D$. Here $\vf^{\mu}$ is a vector field induced by an equivariant map $\mu: M \to \kk$, and $c$ denotes the Clifford action. Braverman assumed the set of zeroes of $\vf^{\mu}$ to be compact, and showed that his index then equals the ones used  in \cite{Zhang14, Paradan11} for symplectic manifolds. 

In this paper, we first generalise Braverman's index in the case of $\Spinc$-Dirac operators. We take the map $\mu$ to be a \myemph{$\Spinc$-moment map}, which directly generalises \eqref{eq mom map intro} (where $L$ is now the determinant line bundle of an equivariant $\Spinc$-structure).
Rather than assuming $\vf^{\mu}$ to vanish in a compact set, we assume that $\mu$ is proper, which we show to be a weaker condition. The index takes values in $\hat R(K)$, the Grothendieck group of the semigroup of representations of $K$ in which all irreducible representations occur with finite multiplicities. We denote this index by
\[
\indK(\cS, \mu) \quad \in \hat R(K),
\]
where $\cS \to M$ is the spinor bundle. 

Our main result is that it satisfies the quantisation commutes with reduction principle. In this context, analogously to \eqref{eq def quant intro}, we also denote the index by 
\[
\QSpinc_K(M, \mu) :=  \indK(\cS, \mu) \quad \in \hat R(K).
\]
Let $\kk^M$ be the generic (i.e.\ minimal) infinitesimal stabiliser of the action by $K$ on $M$. In the symplectic case, there is always an element $\xi \in \kk$ whose stabiliser $\kh \subset \kk$ with respect to the adjoint action satisfies
\[
([\kh, \kh]) = ([\kk^M, \kk^M]),
\]
where the round brackets denote conjugacy classes. (See Remark 3.12 in \cite{Lerman98}). In the $\Spinc$-case, existence of such a stabiliser algebra $\kh$ is necessary for $\QSpinc_K(M, \mu)$ to be nonzero.
\begin{theorem}[Quantisation commutes with reduction] \label{thm [Q,R]=0 intro}
Let $K$ be a compact, connected Lie group, and let $M$ be an even-dimensional, connected 
manifold, with an
 action by $K$, and a  $K$-equivariant $\Spinc$-structure. Let $\mu$ be a $\Spinc$-moment map, 
 and suppose it is proper.  
 If there is no algebra $\kh \subset \kk$ as above, then  $\QSpinc_K(M, \mu) = 0$. If such an algebra does exist,
 then the multiplicities $m_{\lambda} \in \Z$ in 
\begin{equation} \label{eq [Q,R]=0 intro 2}
 Q^{\Spin^c}_K(M, \mu) = \bigoplus_{\lambda} m_{\lambda} \pi_{\lambda} \quad \in \hat R(K),
\end{equation}
equal sums of quantisations of reduced spaces, as specified in \eqref{eq [Q,R]=0}. If the generic stabiliser of the action is Abelian, which in particular occurs if $\mu$ has a regular value, then this sum has a single term:
\[
m_{\lambda} =  Q^{\Spinc}(M_{\lambda} ).
\]
\end{theorem}
This result is a generalisation of Paradan and Vergne's result in \cite{Paradan14} from compact to noncompact manifolds, and a $\Spinc$-analogue of the results for symplectic manifolds in \cite{Zhang14, Paradan11}.

\subsection*{Ingredients of the proof}

We use a combination of analytic and geometric methods. To prove that our index is well-defined, we show that every irreducible representation occurs in it with finite multiplicity. This proof is based on a vanishing result. 
\begin{theorem}[Vanishing]\label{thm vanishing intro}
For every irreducible representation $\pi \in \hat K$, there is a constant $C_{\pi} > 0$ such that for every even-dimensional $K$-equivariant $\Spinc$-manifold $U$, with spinor bundle $\cS \to U$ and moment map $\mu$ such that $\vf^{\mu}$ vanishes in a compact set, one has
\[
\bigl[\indK(\cS, \mu):\pi \bigr] = 0
\]
if $\|\mu(m)\| > C_{\pi}$ for all $m \in U$.
\end{theorem}
(Here $[V:\pi]$ denotes the multiplicity of $\pi$ in $V \in \hat R(K)$.) 

The proof of this result, given in Section \ref{sec vanishing}, is inspired by Tian and Zhang's \cite{Zhang98} analytic proof that quantisation commutes with reduction for compact symplectic manifolds. A geometric deformation of the moment map $\mu$ allows us to generalise their estimates from compact symplectic manifolds to noncompact $\Spinc$-manifolds. An interesting aspect of this proof is that a function $d$, which plays a central role in \cite{Paradan14}, emerges in a completely different way.

An important property of the index we define, and the key step in reducing the proof of Theorem \ref{thm [Q,R]=0 intro} to the compact case, is that it is multiplicative in a suitable sense.
Let $M$ and $N$ be  connected, even-dimensional, $K$-equivariant $\Spinc$-manifolds, with spinor bundles $\cS_M \to M$ and $\cS_N \to N$, and moment maps $\mu_M: M \to \kk^*$ and $\mu_N: N \to \kk^*$. Suppose that $\mu_M$ is proper, and that $N$ is \emph{compact}. 
Let $\ind_K(\cS_N)$ be the usual equivariant index of the $\Spinc$-Dirac operator on $N$. 
\begin{theorem}[Multiplicativity] \label{thm mult intro}
We have
\begin{equation} \label{eq mult intro}
\indK(\cS_{M\times N}, \mu_{M\times N}) = \indK(\cS_M, \mu_M) \otimes \ind_K(\cS_N) \quad \in \hat R(K).
\end{equation}
\end{theorem}

In the compact symplectic case, one can deduce the decomposition \eqref{eq [Q,R]=0 intro} from the case for the trivial representation, i.e.\  $\lambda = 0$. This is based on a multiplicativity property as in Theorem \ref{thm mult intro}, which is much simpler for compact manifolds.
For noncompact symplectic manifolds and proper moment maps, the main difficulty in the proofs in \cite{Zhang14, Paradan11} that quantisation commutes with reduction is to prove Theorem \ref{thm mult intro} for symplectic manifolds $M$. After that, one applies the case for reduction at the trivial representation as for compact symplectic manifolds.

Analogously to the symplectic case, Theorem \ref{thm mult intro} plays an important role in the proof of Theorem \ref{thm [Q,R]=0 intro}. In this paper, we prove Theorem \ref{thm mult intro}, and use it to obtain a localised expression for the multiplicities $m_{\lambda}$ in \eqref{eq [Q,R]=0 intro 2}. This expression is the same as the one in \cite{Paradan14}, so that Paradan and Vergne's arguments can be applied directly from that point onwards. This leads to a proof of Theorem \ref{thm [Q,R]=0 intro}.

To prove Theorem \ref{thm mult intro} in the general $\Spinc$-setting, we apply a localised version of cobordism invariance of Braverman's index. This is done in Section \ref{sec mult}. Defining a useful notion of cobordism for possibly noncompact manifolds is nontrivial. Braverman succeeded in finding such a definition, and in proving that his index is invariant under this notion of cobordism. We cannot directly apply this cobordism invariance to prove Theorem \ref{thm mult intro}, however, because the vector fields appearing in the arguments do not have compact sets of zeroes. (An essential assumption in Braverman's definition.) But by considering the multiplicity of any fixed irreducible representation in both sides of \eqref{eq mult intro}, we are able to localise the problem to sets where it is possible to construct a cobordism.
This construction involves a procedure to replace a given moment map $\mu$, on an open set $U$ where the vector field $\vf^{\mu}$ has a compact set of zeroes, by another moment map $\tilde \mu$ that is \emph{proper} on $U$, without changing the resulting index.

The proofs in \cite{Zhang14, Paradan11} of Theorem \ref{thm mult intro} in the symplectic case are quite involved. 
While this is also a matter of taste, to the authors
the cobordism argument in Section \ref{sec mult} seems simpler. In addition, 
 it applies directly to the $\Spinc$-case. In \cite{HochsSong15b}, we apply analogous arguments to give a short proof in the sympectic case.

\subsection*{Acknowledgements}

The first author is grateful to Paul-\'Emile Paradan for his hospitality during a visit to Montpellier in May 2014, and for his explanations of
 his work with Vergne. He would also like to thank Eckhard Meinrenken for inviting him to visit the University of Toronto in August 2014. He was supported by Marie Curie fellowship PIOF-GA-2011-299300 from the European Union. The second author would like to thank Eckhard Meinrenken and Nigel Higson for many beneficial discussions.  The authors are grateful to the referees for some useful suggestions and corrections, and in particular for suggesting a shorter proof of Lemma \ref{lem Z phi cpt}.

\subsection*{Notation and conventions}

All manifolds, maps, vector bundles and group actions are tacitly assumed to be smooth, unless stated otherwise. 

The space of vector fields on a manifold $M$ will be denoted by $\cX(M)$, and the space of $k$-forms by $\Omega^k(M)$. 
If $E \to M$ is a smooth vector bundle, then $\Gamma^{\infty}(E)$ is its space of smooth sections, and $\Gamma^{\infty}_c(E) \subset \Gamma^{\infty}(E)$ is the subspace of compactly supported sections.


\section{An analytic index}

In \cite{Braverman02}, Braverman defined an equivariant  index of Dirac-type operators for compact Lie groups acting on possibly noncompact manifolds. In Section \ref{sec Spinc Dirac}, we will generalise this index in the case of $\Spinc$-Dirac operators, and state its properties that we will prove in the rest of this paper. In the current section, we review the material from \cite{Braverman02} we will use.

Throughout this paper, $K$ will be a compact, connected Lie group, with Lie algebra $\kk$, acting on a manifold $M$. 
We will consider a $K$-invariant Riemannian metric $g$ on $M$. In this section, we assume that $M$ is complete with respect to this metric. 

\subsection{Deformed Dirac operators} \label{sec deformed Dirac}

 Let $\cE = \cE^+ \oplus \cE^- \to M$ be a $\Z_2$-graded,  complex vector bundle, equipped with a Hermitian metric. Let 
 \[
 c: TM \to \End(\cE)
 \]
be a vector bundle homomorphism, whose image lies in the skew-adjoint, odd endomorphisms, and such that for all $v \in TM$,
\[
c(v)^2 = -\|v\|^2.
\]
Then $\cE$ is called a \myemph{Clifford module} over $M$, and $c$ is called the \myemph{Clifford action}.

A \myemph{Clifford connection} is a Hermitian connection $\nabla^{\cE}$ on $\cE$ that preserves the grading on $\cE$, such that for all vector fields $v, w \in \cX(M)$,
\[
[\nabla^{\cE}_v, c(w)] = c(\nabla^{TM}_v w),
\]
where $\nabla^{TM}$ is the Levi--Civita connection on $TM$. We will identify $TM \cong T^*M$ via the Riemannian metric. Then the Clifford action $c$ defines a map
\[
c: \Omega^1(M; \cE) \to \Gamma^{\infty}(\cE). 
\]
The \myemph{Dirac operator} $D$ associated to a Clifford connection $\nabla^{\cE}$ is defined as the composition
\begin{equation} \label{eq def Dirac}
D: \Gamma^{\infty}(\cE) \xrightarrow{\nabla^{\cE}} \Omega^1(M; \cE) \xrightarrow{c} \Gamma^{\infty}(\cE).
\end{equation}
In terms of a local orthonormal frame $\{e_{1}, \ldots, e_{\dim M}\}$, one has
\begin{equation} \label{eq Dirac local}
D = \sum_{j=1}^{\dim M} c(e_j) \nabla^{\cE}_{e_j}.
\end{equation}
This operator interchanges sections of $\cE^+$ and $\cE^-$. We will denote the restriction of $D$ to $\Gamma^{\infty}(\cE^{\pm})$ by  $D^{\pm}$. 

Suppose that $\cE$ is a $K$-equivariant vector  bundle, the action by $K$ preserves the grading on $\cE$,  and 
for all $k \in K$, $m \in M$, $v \in T_mM$ and $e \in \cE_m$ we have\footnote{In fact, this condition implies that the action by $K$ preserves the Riemannian metric.}
\[
k\cdot  c(v)e = c(k\cdot v) k\cdot e.
\]
Then $\cE$ is called a \myemph{$K$-equivariant Clifford module} over $M$. 
Let $\nabla^{\cE}$ be a $K$-invariant Clifford connection on $\cE$. Then the Dirac operator $D$ associated to $\nabla^{\cE}$ is $K$-equivariant. 

If $M$ is compact, the kernel of $D$ is finite-dimensional. Then one has the equivariant index of $D^+$, 
\begin{equation} \label{eq index cpt}
\ind_K(\cE) := \ind_K(D^+) = [\ker D^+] - [\ker D^-] \quad \in R(K),
\end{equation}
where $R(K)$ is the representation ring of $K$. Braverman defined an equivariant index without assuming $M$ to be compact.

\subsection{Braverman's index}

To define an index of Dirac operators on noncompact manifolds, Braverman used a smooth, $K$-equivariant map
\[
\varphi: M \to \kk.
\]
Such a map induces a vector field $\vf^{\varphi} \in \cX(M)$, defined by
\begin{equation} \label{eq def V phi}
\vf^{\varphi}_m = \ddt \exp(-t\varphi(m))m.
\end{equation}
\begin{definition} \label{def deformed Dirac}
The \myemph{Dirac operator deformed by $\varphi$} is the operator
\[
D_{\varphi} := D - \sqrt{-1}c(\vf^{\varphi})
\]
on $\Gamma^{\infty}(\cE)$.
\end{definition}
\begin{remark}
The vector field $\vf^{\varphi}$ equals minus the vector field used by Braverman (see (2.2) in \cite{Braverman02}). This leads to the minus sign in the definition of the deformed Dirac operator, which is not present in (2.6) in \cite{Braverman02}. The minus sign in the definition of vector fields induced by Lie algebra elements, used in the definition of
 $\vf^{\varphi}$, is needed to make $\Spinc$-moment maps well-defined in Subsection \ref{sec moment maps}.
\end{remark}

We will denote the set of zeroes of $\vf^{\varphi}$ by $Z_{\varphi}$. It is important for the definition of Braverman's index that $Z_{\varphi}$ is compact. In fact, a large part of the work in this paper is done to handle cases where $Z_{\varphi}$ may be noncompact.
\begin{definition}
A \myemph{taming map} is an equivariant map $\varphi: M \to \kk$, with the property that $Z_{\varphi}$ is compact. \end{definition}
Let $\varphi$ be a taming map.

Another ingredient used in the definition of the index is a $K$-invariant, nonnegative smooth function $f \in C^{\infty}(M)^K$, that grows fast enough. More precisely, $f$ is required to satisfy
\[
\lim_{m \to \infty} \frac{f(m) \|\vf^{\varphi}_m\|^2}{\|d_m f\| \|\vf^{\varphi}_m\| + f(m)\zeta(m) + 1} = \infty,
\]
where $\zeta$ is a function on $M$, defined in (2.4) in \cite{Braverman02}. By Lemma 2.7 in \cite{Braverman02},  a function $f$ with these properties always exists. Fix such a function $f$. 

Let $\ker^{L^2}(D^{\pm}_{f\varphi})$ be the kernel of the deformed Dirac operator $D^{\pm}_{f\varphi}$, intersected with the space of $L^2$-sections of $\cE$, with respect to the Riemannian density on $M$. The definition of Braverman's index is based on Theorem 2.9 in \cite{Braverman02}, which is the following statement.
\begin{theorem} \label{thm Br index}
Any irreducible representation $\pi$ of $K$  has finite multiplicity $m^{\pm}_{\pi}$ in $\ker^{L^2}(D^{\pm}_{f\varphi})$. The integers $m^{+}_{\pi} - m^-_{\pi}$ do not depend on the choice of the function $f$ and the connection $\nabla^{\cE}$. 
\end{theorem}
This allows one to define an equivariant index for the pair $(\cE, \varphi)$.
\begin{definition} \label{def index Br}
The \myemph{equivariant $L^2$-index} of the pair $(\cE, \varphi)$ is
\[
\ind^{L^2}_K(\cE, \varphi) := \sum_{\pi \in \hat K} (m^+_{\pi} - m^-_{\pi}) \pi \quad \in \hat R(K).
\]
\end{definition}
In this definition, $\hat K$ is the unitary dual of $K$, the numbers $m^{\pm}_{\pi}$ are as in Theorem \ref{thm Br index}, and $\hat R(K)$ is the Grothendieck group of the semigroup of representations of $K$ in which all irreducible representations occur with finite multiplicities, which is  equal to $\Hom_{\Z}(R(K), \Z)$. 

The index of Definition \ref{def index Br} depends on $\varphi$ in general, but see 
Proposition \ref{prop excision}
below, and Lemma 3.16 in \cite{Braverman02}.

\subsection{An index theorem} \label{sec index thm}

Braverman's analytic index of Definition \ref{def index Br} equals a topological index used in the statement of Vergne's conjecture \cite{Vergne06}, and in \cite{Paradan01, Paradan11, Paradan14}. It also equals an analytic index used in \cite{Zhang14} in the context of geometric quantisation.

The topological index is defined as follows. Let $\cE$ and $\varphi$ be as in Subsection \ref{sec deformed Dirac}. Consider the symbol $\sigma_{\varphi}: TM \to \End(\cE)$ defined by
\begin{equation} \label{eq deformed Thom}
\sigma_{\varphi}(v) := \sqrt{-1}c(v - \vf^{\varphi}_m),
\end{equation}
for all $m \in M$ and $v \in T_mM$. This symbol defines a class
\[
[\sigma_{\varphi}] \in K^0(T_KM),
\]
where $T_KM$ is the space of tangent vectors to $M$ orthogonal to $K$-orbits. By embedding a $K$-invariant, relatively compact
neighbourhood of $Z_{\varphi}$ into a compact manifold, one can apply Atiyah's index of transversally elliptic symbols \cite{Atiyah74} to obtain
\begin{equation} \label{eq top index}
\ind_K^{\topp}[\sigma_{\varphi}] \in \hat R(K).
\end{equation}
Theorem 5.5 in \cite{Braverman02} states that
\begin{equation} \label{eq index thm}
\ind_K^{\topp}[\sigma_{\varphi}] = \indK(\cS, \varphi). 
\end{equation}

In Section 1.4 of \cite{Zhang14}, Ma and Zhang show that their APS-style definition of geometric quantisation equals the above two indices as well (in cases where $Z_{\varphi}$ is compact.)

Braverman's index theorem \eqref{eq index thm} implies that the analytic index of Definition \ref{def index Br} has all properties of the topological index \eqref{eq top index}. One of these is a multiplicativity property. Let $N$ be a compact, connected Riemannian manifold, with a $K$-equivariant Clifford module $\cE_N \to N$. Denote the Clifford module on $M$ by $\cE_M$ for clarity, and let 
 $\cE_{M\times N} = \cE_M \boxtimes \cE_N \to M \times N$ be the product Clifford module. In terms of the Clifford actions 
$c_M: TM \to \End(\cE_M)$ and $c_N: TN \to \End(\cE_N)$, the Clifford action $c_{M\times N}: T(M \times N) \to \End(\cE_{M\times N})$ is defined by
\[
c_{M\times N}(v, w) = c_M(v) \otimes 1_{\cE_N} + \gamma_M \otimes c_N(w),
\]
for $v \in TM$ and $w \in TN$, where $\gamma_M$ is the grading operator on $\cE_M$.  Let $\hat \varphi: M \times N \to \kk^*$ be the pullback of the taming map $\varphi$ along the projection map to $M$. 
 \begin{theorem} \label{thm mult index}
We have
\[
\indK(\cE_{M\times N}, \hat \varphi) = \indK(\cS_M, \varphi) \otimes \ind_K(\cE_N) \quad \in \hat R(K). 
\]
\end{theorem}
\begin{proof}
Consider  a function $f \in C^{\infty}(M)$, with pullback $\hat f$ to $M \times N$. Let
$\sigma_{f\varphi}$ be the Clifford action on $M$ deformed by $f \varphi$ as in \eqref{eq deformed Thom}, and let 
 $\sigma_{\hat f \hat \varphi}$ be the Clifford action  on $M \times N$ deformed by $\hat f \hat \varphi$. Then
 $\sigma_{\hat f \hat \varphi}$
  equals the product, in the sense of  Theorem 3.5 in \cite{Atiyah74}, of the  $\sigma_{f\varphi}$  and the Clifford action on $N$. Therefore, Theorem 3.5 in \cite{Atiyah74} implies the claim.

Alternatively, one can decompose the $\Spinc$-Dirac operator $D_{M \times N}$ on $M \times N$ into the Dirac operators $D_M$ on $M$ and $D_N$ on $N$ as
\[
D_{M\times N} = D_M \otimes 1_{\cE_N} + \gamma_M \otimes D_N,
\]
and deduce from this that
\[
\left( D_{M\times N} - \ii c(\hat f \vf^{\hat \varphi}) \right)^2 = \bigl(D_M - \sqrt{-1}c(f \varphi) \bigr)^2 \otimes 1_{\cE_N} + 
	1_{\cE_M} \otimes D_N^2.
\]
Since all squared operators are nonnegative, it follows that
 the $L^2$-kernel of $D_{M \times N} - \sqrt{-1}c(\hat f \vf^{\hat \varphi})$ equals
\[
\ker^{L^2}\left( D_{M\times N} - \ii c(\hat f \vf^{\hat \varphi}) \right) = 
\ker^{L^2}\bigl(D_M - \sqrt{-1}c(f \varphi) \bigr) \otimes \ker(D_N).
\]
This includes the appropriate gradings, so the claim follows. 
\end{proof}

We will define an equivariant index of $\Spinc$-Dirac operators in Section \ref{sec Spinc Dirac}, based on Braverman's index. In the last step of the proof of our main result, we will use facts from \cite{Paradan14} about the topological index defined in \eqref{eq top index}. There, we will tacitly use the index theorem \eqref{eq index thm}.

\subsection{Cobordism invariance}

To define a meaningful notion of cobordism for noncompact manifolds, one has to include more information than just the manifolds themselves. (Otherwise any manifold $M$ is cobordant to the empty set, through the cobordism $M \times [0, 1[$.) Braverman defined a notion of cobordism that includes taming maps.
 A fundamental, and very useful, property of his index is invariance under this version of cobordism. This will play a key role in our arguments in Subsection \ref{sec mult invar}.

Braverman's cobordism is defined in Definitions 3.2 and 3.5 in \cite{Braverman02}. Cobordism invariance of his index is Theorem 3.7 in \cite{Braverman02}.
We will only need a special case of this cobordism invariance, where the two manifolds and Clifford modules in question are equal. This involves the notion of a  homotopy of taming maps.  
\begin{definition} \label{def htp}
Let $\varphi_1, \varphi_2: M \to \kk$ be two taming maps. A \myemph{homotopy of taming maps} between $\varphi_1$ and $\varphi_2$ is a taming map $\varphi: M \times [0,1] \to \kk$, such that, for some $\varepsilon \specialin ]0, \frac{1}{2}[$,
\[
\varphi|_{[0, \varepsilon[} = \varphi_1 \otimes 1_{[0, \varepsilon[},
\]
and
\[
\varphi|_{]1-\varepsilon, 1]} = \varphi_2 \otimes 1_{]1-\varepsilon, 1]}.
\]
If such a homotopy exists, then $\varphi_1$ and $\varphi_2$ are called \myemph{homotopic}.
\end{definition}
\begin{theorem}[Homotopy invariance] \label{thm htp invar}
If two taming maps $\varphi_1$ and $\varphi_2$ are homotopic, then
\[
\indK(\cS, \varphi_1) = \indK(\cS, \varphi_2). 
\]
\end{theorem}
\begin{proof}
A homotopy of taming maps defines a cobordism in the sense of Definition 3.5 in \cite{Braverman02}. Therefore, the claim follows from Theorem 3.7 in \cite{Braverman02}.
\end{proof}

\begin{remark}
In Definition \ref{def htp}, it is essential that the map $\varphi$ is taming, i.e.\ $Z_{\varphi} = 0$. For example, the linear path between two taming maps need not be taming itself, so that certainly not all taming maps are homotopic. A large part of the work in Section \ref{sec mult} is to show that a certain path between two taming maps is in fact taming itself, so that it defines a homotopy of taming maps.
\end{remark}

While the only kind of cobordisms we will use directly are homotopies of taming maps, we will also need some consequences of cobordism invariance of Braverman's index in the general sense of Section 3 of \cite{Braverman02}. One of these
is 
%
%
the following vanishing result.
\begin{lemma} \label{lem vanishing}
If $Z_{\varphi} = \emptyset$, then
\[
\ind_{K}^{L^2}(\cE, \varphi) = 0.
\]
\end{lemma}
\begin{proof}
See Lemma 3.12 in \cite{Braverman02}.
\end{proof}

\subsection{Non-complete manifolds} \label{sec index noncomplete}

In our localisation arguments, we will often consider an extension of the index of Definition \ref{def index Br} to non-complete manifolds. This can be defined using the following arguments, analogous to those in Section 4.2 of \cite{Braverman02}. 

First, let $(U, g)$ be a Riemannian manifold, equipped with an isometric action by $K$, a $K$-equivariant Clifford module $\cE \to U$, and a taming map $\varphi$. Suppose there is a $K$-invariant neighbourhood  $V$ of $Z_{\varphi}$, and a $K$-invariant, positive function $\chi \in C^{\infty}(U)^K$, such that $\chi|_V \equiv 1$, and $U$ is complete in the Riemannian metric
\[
g^{\chi} := {\chi^2} g.
\]
If $U$ is a $K$-invariant, relatively compact subset of a Riemannian manifold with an isometric action by $K$, and $\partial U$ is a smooth hypersurface, such a function can be constructed as in Section 4.2 of \cite{Braverman02}. 

Then  $\cE$ becomes a $K$-equivariant Clifford module over $U$, with respect to the Clifford action
\[
\chi c: TU\to \End(\cE).
\]
This allows us to define the index
\begin{equation} \label{eq index g chi}
\indK(\cE, \varphi, g^{\chi}) \quad \in \hat R(K),
\end{equation}
where the Riemannian metric was added to the notation to emphasise which one is used. 

Now let $M$, $\cE$ and $\varphi$ be as before. In particular, $M$ is complete. Using cobordism invariance of Braverman's index, one obtains the following additivity and excision properties of the index.
\begin{proposition}[Additivity] \label{prop additivity}
Let $\Sigma \subset M$ be a relatively compact, $K$-invariant, smooth hypersurface, on which $\vf^{\varphi}$ does not vanish. Suppose $U_1, U_2 \subset M$ are disjoint $K$-invariant open subsets such that
\[
M \setminus \Sigma = U_1 \cup U_2.
\]
Then
\[
\indK(\cE, \varphi) = \indK(\cE|_{U_1}, \varphi|_{U_1}, g|_{U_1}^{\chi_1}) + \indK(\cE|_{U_2}, \varphi|_{U_2}, g|_{U_2}^{\chi_2}), 
\]
where, for $j = 1, 2$, $\chi_j$ is a function as above, so that $\chi_j \equiv 1$ in a neighbourhood of ${U_j \cap Z_{\varphi}} $, and $U_j$ is complete in the metric $, g|_{U_j}^{\chi_j}$.
\end{proposition}
\begin{proof}
See Corollary 4.7 in \cite{Braverman02}.
\end{proof}
Proposition \ref{prop additivity} and Lemma \ref{lem vanishing} imply the following excision property of the index.
\begin{proposition}[Excision] \label{prop excision}
Let ${U'} \subset M$ be a relatively compact, $K$-invariant neighbourhood of $Z_{\varphi}$ such that $\partial {U'}$ is a smooth hypersurface in $M$. Then
\[
\indK(\cE, \varphi) = \indK(\cE|_{U'}, \varphi|_{U'}, g|_{{U'}}^{\chi}), 
\]
with $\chi$ as above.
\end{proposition}
One consequence of this excision property is that the index \eqref{eq index g chi} is independent of the function $\chi$. Another is that the following definition is a well-defined extension of Definition \ref{def index Br}.
\begin{definition}
\label{def index noncomplete}
Let $(U, g)$ be a (possibly non-complete) Riemannian manifold, equipped with an isometric action by $K$, a $K$-equivariant Clifford module $\cE \to U$, and a taming map $\varphi$.
Let $U'$ be a $K$-invariant, relatively compact open neighbourhood $Z_{\varphi}$, such that $\partial U'$ is a smooth hypersurface in $U$. Within $U'$, we can choose a function $\chi$ as above. Then the \emph{equivariant $L^2$-index} of the pair $(\cE, \varphi)$ is
\[
\ind_{K}^{L^2}(\cE, \varphi) := \indK(\cE|_{U'}, \varphi|_{U'}, g|_{U'}^{\chi}) \quad \in \hat R(K).
\]
\end{definition}

%
From now on, we will not assume manifolds to be complete, but apply Definition \ref{def index noncomplete} where necessary.


\section{$\Spinc$-Dirac operators and proper moment maps} \label{sec Spinc Dirac}

We now specialise to the case of $\Spinc$-Dirac operators. In that setting, we define a generalisation of Braverman's index, for a natural class of maps $\varphi$, without assuming  $Z_{\varphi}$ to be compact. This assumption is replaced by properness of the $\Spinc$-moment map defined in Subsection \ref{sec moment maps}. In Subsection \ref{sec [Q,R]=0}, we state the main result of this paper, that this index of $\Spinc$-Dirac operators decomposes into irreducible representations according to the {quantisation commutes with reduction} principle. 

As before, let $M$ be a 
Riemannian 
manifold, on which a compact, connected Lie group $K$ acts
 isometrically. 
From now on, we will suppose that $M$ is even-dimensional. 

\subsection{$\Spinc$-Dirac operators} \label{subsec Spinc Dirac}

Suppose $M$ has a $K$-equivariant $\Spinc$-structure. (Then we will call $M$  a $K$-equivariant $\Spinc$-manifold.) By definition, this means that there is a $\Z_2$-graded, $K$-equivariant complex vector bundle $\cS \to M$, called the \myemph{spinor bundle},
and a $K$-equivariant isomorphism
\[
c: \Cl(TM) \xrightarrow{\cong} \End(\cS)
\]
of graded algebra bundles, where $\Cl(TM)$ is the complex Clifford bundle of $TM$.
Then $\cS$ is a $K$-equivariant Clifford module over $M$.  The \myemph{determinant line bundle} associated to the spinor bundle $\cS$ is the line bundle
\[
L := \Hom_{\Cl(TM)}(\overline{\cS}, \cS) \to M,
\]
where $\overline{\cS} \to M$ is the vector bundle $\cS$, with the opposite complex structure. See e.g.\ Appendix D of \cite{Guillemin98} or Appendix D of \cite{Lawson89}  for more details on $\Spinc$-structures.

Locally, on small enough open subsets $U$ of $M$, one has
\begin{equation} \label{eq decomp S}
\cS|_U \cong \cS^U_0 \otimes L|_U^{1/2},
\end{equation}
where $\cS_0^U \to U$ is the spinor bundle of a local $\Spin$-structure. The Levi--Civita connection on $TU$ induces a connection $\nabla^{\cS^U_0}$ on $\cS^U_0$. Fix a $K$-invariant Hermitian connection $\nabla^L$ on $L$. 
Together with $\nabla^{\cS^U_0}$, this induces a connection $\nabla^{\cS|_U}$ on $\cS|_U$, via the decomposition \eqref{eq decomp S}:
\[
\nabla^{\cS|_U} := \nabla^{\cS^U_0} \otimes 1_{L|_U^{1/2}} + 1_{\cS^U_0} \otimes \nabla^{L|_U^{1/2}}.
\]
Here $\nabla^{L|_U^{1/2}}$ is the connection on $L|_U^{1/2}$ induced by $\nabla^L$.
These local connections combine to a globally well-defined connection $\nabla^{\cS}$ on $\cS$ (see e.g.\ Proposition D.11 in \cite{Lawson89}).

The \myemph{$\Spinc$-Dirac operator} associated to $\nabla^L$ is the operator $D$ on $\Gamma^{\infty}(\cS)$ defined as in \eqref{eq def Dirac}, with $\cE$ replaced by $\cS$, and $\nabla^{\cE}$ by $\nabla^{\cS}$.

\subsection{Moment maps} \label{sec moment maps}

An important role will be played by the $\Spinc$-moment map associated to the connection $\nabla^L$ on $L$ chosen in Subsection \ref{subsec Spinc Dirac}. This generalises the moment map in symplectic geometry.

In this subsection only, we consider a more general situation. Let $L \to M$ be any $K$-equivariant line bundle, and let $\nabla^L$ be any $K$-invariant connection on $L$.
For any element $X \in \kk$, we denote the induced vector field on $M$ by $X^M$, i.e.\ 
\begin{equation} \label{eq def XM}
X^M_m := \ddt \exp(-tX)\cdot m,
\end{equation}
for all $m \in M$. In addition, for any $K$-equivariant vector bundle $E \to M$, and any $X \in \kk$, we 
write  $\cL^E_X$ for the Lie derivative of smooth sections of $E$ with respect to $X$. 
\begin{definition} \label{def mom map}
The \myemph{moment map} associated to $\nabla^L$ is the map
\[
\mu: M \to \kk^*
\]
defined by
\[
2 \ii \mu_X = \cL^L_X - \nabla^L_{X^M} \quad \in \End(L) \cong C^{\infty}(M, \C),
\]
for all $X \in \kk$. Here $\mu_X \in C^{\infty}(M)$ is the pairing of $\mu$ and $X$.

If $L$ is the determinant line bundle of a $K$-equivariant $\Spinc$-structure, then $\mu$ is called a \emph{$\Spinc$-moment map}.
\end{definition}
One can compute that for all $X \in \kk$, 
\begin{equation} \label{eq moment curv}
2 \ii d\mu_X = R^{\nabla^L}(X^M, \relbar),
\end{equation}
where $R^{\nabla^L}$ is the curvature of $\nabla^L$ (see e.g.\ Lemma 2.2 in \cite{Mathai14}). This implies that $\mu$ is a moment map in the usual symplectic sense if the closed two-form $R^{\nabla^L}$ is nondegenerate.
A direct consequence of \eqref{eq moment curv} is the following important property of moment maps.
\begin{lemma} \label{lem mu loc const}
Let $H<K$ be a Lie subgroup, with Lie algebra $\kh$. Then the composition of a 
moment map $\mu$ with the restriction map from $\kk^*$ to $\kh^*$ is locally constant on the fixed point set $M^H$.
\end{lemma}

Fix an $\Ad(K)$-invariant inner product on $\kk$. From now on, we will use this to identify $\kk^* \cong \kk$, and in particular view $\mu$ as a map to $\kk$. Then we have the vector field $\vf^{\mu}$ on $M$, defined as in \eqref{eq def V phi}.
Suppose that $\mu$ is  \myemph{proper}. 
We will see in Proposition \ref{prop mu proper} that a taming moment map can always be replaced by a proper one, without changing the resulting index. Therefore, assuming $\mu$ to be proper is a weaker assumption than assuming $\mu$ to be taming.
Because $\mu$ is proper, the set $Z_{\mu} \subset M$ where $\vf^{\mu}$ vanishes can be decomposed in a way that allows us to define an suitable index of $\Spinc$-Dirac operators.

Let $T<K$ be a maximal torus, with Lie algebra $\kt$. Let $\kt^*_+ \subset \kt^*$ be a choice of closed positive Weyl chamber.
\begin{lemma} \label{lem decomp Z mu}
There is a subset $\Gamma \subset \kt^*_+$ such that
\[
Z_{\mu} = \bigcup_{\alpha \in \Gamma} K\cdot  \left( M^{\alpha} \cap \mu^{-1}(\alpha)\right),
\]
and for all $R>0$, there are finitely many $\alpha \in \Gamma$ with $\|\alpha\| \leq R$.
\end{lemma}
Note that the sets $K\cdot  \left( M^{\alpha} \cap \mu^{-1}(\alpha)\right)$ are compact by properness of $\mu$.
\begin{proof}
Consider the subset
$
Y := \mu^{-1}(\kt^*_+) \subset M.
$
Then $K\cdot Y = M$. Since $Z_{\mu}$ is $K$-invariant, we have
$
Z_{\mu} = K\cdot (Z_{\mu} \cap Y).
$
And because $\mu(Y) \subset \kt^*$, 
\[
Z_{\mu} \cap Y = \bigcup_H Y^H \cap \mu^{-1}(\kh),
\]
where $H$ runs over the stabilisers of the action by $T$ on $Y$.

Fix $R > 0$. By properness of $\mu$, the set
\[
Y_R := \{m \in Y; \|\mu(m)\| < R\}
\]
is relatively compact. Hence the action by $T$ on $Y_R$ has finitely many stabilisers $H_1, \ldots, H_{k_R}$. Now
\[
Z_{\mu} \cap Y_R = \bigcup_{j=1}^{k_R} Y^{H_j}_R \cap \mu^{-1}(\kh_j).
\]
Because of Lemma \ref{lem mu loc const}, the map $\mu$ is locally constant on the sets $M^{H_j} \cap \mu^{-1}(\kh_j)$, hence also on their subsets
$Y^{H_j}_R \cap \mu^{-1}(\kh_j)$. For a connected component $F$ of $Y^{H_j}_R \cap \mu^{-1}(\kh_j)$, let $\alpha_F \in \kh_j \cap \kt^*_+$ be single value of $\mu$ on $F$. Such an element $\alpha_F$ is a weight of the action by $H_j$ on $L|_F$, and hence lies on an integral lattice. Therefore, for fixed $j$, the set
\[
\Gamma_{R, j} := \{\alpha_F; \text{$F$ a connected component of $Y^{H_j}_R \cap \mu^{-1}(\kh_j)$}\}
\]
is finite. So the set
\[
\Gamma_R := \bigcup_{j=1}^{k_R} \Gamma_{R, j}
\]
is finite as well. The claim follows, with
\[
\Gamma := \bigcup_{R>0} \Gamma_{R}.
\]
\end{proof}
The equality (2.16) in \cite{Paradan11} is a symplectic version of this lemma.
In the compact case, 
see Lemma 3.15 in \cite{Kirwan-book} for symplectic manifolds, and Lemma 2.9 in \cite{Paradan14} for $\Spinc$-manifolds.



\subsection{The index for proper moment maps} \label{sec index proper}


We now return to the situation of Subsection \ref{subsec Spinc Dirac}, where $L \to M$ is the determinant line bundle of a $K$-equivariant $\Spinc$-structure. Let $\mu$ be the $\Spinc$-moment map associated to the chosen connection $\nabla^L$ on $L$. Suppose $\mu$ is proper.

We will use Lemma \ref{lem decomp Z mu} to define an equivariant index of $\Spinc$-Dirac operators, for proper moment maps. As in Lemma \ref{lem decomp Z mu}, write
\[
Z_{\mu} = \bigcup_{\alpha \in \Gamma} K\cdot  \left( M^{\alpha} \cap \mu^{-1}(\alpha)\right).
\]
For every $\alpha \in \Gamma$, let $U_{\alpha}$ be a $K$-invariant neighbourhood of $K\cdot  \left( M^{\alpha} \cap \mu^{-1}(\alpha)\right)$. 
We choose these neighbourhoods so small that $U_{\alpha} \cap U_{\beta} = \emptyset$ 
if $\alpha \not= \beta$, and
\begin{equation} \label{eq mu U alpha}
\| \mu(U_{\alpha}) \| \subset \bigl]\|\alpha\|-1, \|\alpha\| + 1 \bigr[.
\end{equation}

The definition of our index is based on the following vanishing result. For $V \in \hat R(K)$ and $\pi \in \hat K$, we will denote the multiplicity of $\pi$ in $V$ by $[V:\pi]$. 
\begin{theorem}[Vanishing]\label{thm vanishing}
For every irreducible representation $\pi \in \hat K$, there is a constant $C_{\pi} > 0$ such that for every even-dimensional $K$-equivariant $\Spinc$-manifold $U$, with spinor bundle $\cS \to U$ and taming moment map $\mu$, one has
\[
\bigl[\indK(\cS, \mu):\pi \bigr] = 0
\]
if $\|\mu(m)\| > C_{\pi}$ for all $m \in U$.
\end{theorem}
In this theorem, $\indK(\cS, \mu)$ is defined as in Definition \ref{def index noncomplete}, since $U$ may not be complete. This result will be proved in Section \ref{sec vanishing}.

Since $\mu$ is proper, the set $Z_{\mu} \cap U_{\alpha}$ is compact for all $\alpha$. Hence the index
\[
\indK(\cS|_{U_{\alpha}}, \mu|_{U_{\alpha}}) \quad \in \hat R(K)
\]
is well-defined as in Definition \ref{def index noncomplete}.
Applying Theorem \ref{thm vanishing} to the open sets $U_{\alpha} \subset M$, we find that only finitely many of these indices contribute to the multiplicity of any given irreducible representation.
\begin{corollary} \label{cor fin mult}
Let $\pi \in \hat K$ be any irreducible representation. Then for all but finitely many $\alpha \in \Gamma$, we have
\[
\bigl[\indK(\cS|_{U_{\alpha}}, \mu|_{U_{\alpha}}):\pi \bigr] = 0.
\]
\end{corollary}
\begin{proof}
For $\pi \in \hat K$, let $C_{\pi}>0$ be as in Theorem \ref{thm vanishing}. By Lemma \ref{lem decomp Z mu}, there are only finitely many $\alpha \in \Gamma$ with $\|\alpha\| \leq C_{\pi} + 1$. For all other $\alpha \in \Gamma$, one has $\|\mu\| > C_{\pi}$ on $U_{\alpha}$ by \eqref{eq mu U alpha}, so 
\[
\bigl[\indK(\cS|_{U_{\alpha}}, \mu|_{U_{\alpha}}):\pi \bigr] = 0.
\]
\end{proof}

Corollary \ref{cor fin mult} allows us to generalise Braverman's index of Definitions \ref{def index Br} and \ref{def index noncomplete} in the following way, for $\Spinc$-Dirac operators and proper moment maps.
\begin{definition} \label{def index proper}
The \myemph{equivariant $L^2$-index} of the pair $(\cS, \mu)$ is
\[
\indK(\cS, \mu) := \sum_{\alpha \in \Gamma} \indK(\cS|_{U_{\alpha}}, \mu|_{U_{\alpha}}).
 \quad \in \hat R(K).
\]
\end{definition}
Proposition \ref{prop excision} implies that this definition is independent of the choice of the sets $U_{\alpha}$.
In addition, that proposition shows that Definition \ref{def index proper} reduces to 
Definition \ref{def index noncomplete} if $\mu$ is taming, and hence to Definition \ref{def index Br} if $M$ is also complete.

\begin{remark} \label{rem deform conn}
Because we are dealing with $\Spinc$-Dirac operators here, the deformed Dirac operator $D_{\mu}$ used to define the index in Definition \ref{def index proper} can be obtained as an undeformed Dirac operator for a different choice of connection on $L$. Indeed, Let $\widetilde{\nabla}^L$ be another $K$-invariant, Hermitian connection on $L$. Write
\[
\widetilde{\nabla}^L = \nabla^L - \ii \alpha,
\]
for a $K$-invariant one-form $\alpha \in \Omega^1(M)^K$. The resulting $\Spinc$-Dirac operator $\widetilde{D}$ then equals
\[
\widetilde{D} = D - \ii c(\alpha).
\]
Taking $\alpha$ to be the one-form dual to the vector field $\vf^{\mu}$, we get $\widetilde{D} = D_{\mu}$. This will indirectly play a role in Subsection \ref{sec mu proper}. 
\end{remark}

\begin{remark}
In the symplectic setting, the ways geometric quantisation for proper moment maps was defined in \cite{Zhang14, Paradan11}, are related to
the way we generalised Definition \ref{def index Br} to Definition \ref{def index proper} in the $\Spinc$-case. In \cite{Zhang14, Paradan11}, the symplectic manifold to be quantised was broken up into relevant pieces, on which an equivariant index could be applied. In \cite{Zhang14}, an APS-type index was used, whereas in \cite{Paradan11}, an index of transversally elliptic symbols was used. The comments in Subsection \ref{sec index thm} imply that Definition \ref{def index proper} reduces  to the definitions of quantisation in \cite{Zhang14, Paradan11} in the symplectic case, modulo a shift in the line bundle used.
\end{remark}

\subsection{Quantisation commutes with reduction: the main result} \label{sec [Q,R]=0}

The main result in this paper is Theorem \ref{thm [Q,R]=0}, which states that the index in Definition \ref{def index proper} satisfies the quantisation commutes with reduction principle. This allows one to determine its decomposition into irreducible representations in a geometric way. Therefore, we will regard the index as the $\Spinc$-quantisation of $(M, \mu)$, as in \cite{Mathai14, Paradan14}, and write
\[
\QSpinc_K(M, \mu) := \indK(\cS, \mu).
\]
This is a slight abuse of notation, because this index depends on the $\Spinc$-structure on $M$ (and possibly not just on the determinant line bundle used to define $\mu$). But in what follows, it will usually be clear which $\Spinc$-structure is used.

We start with a vanishing result.
For any subalgebra $\kh \subset \kk$, let $(\kh)$ be its conjugacy class. Set
\[
\cH_{\kk} := \{ (\kk_{\xi}); \xi \in \kk\}.
\]
Let $(\kk^M)$ be the conjugacy class of the generic (i.e.\ minimal) infinitesimal stabiliser of the action by $K$ on $M$. In the symplectic case, there is always a conjugacy class $(\kh) \in \cH_{\kk}$ such that $([\kk^M, \kk^M]) = ([\kh, \kh])$ (see Remark 3.12 in \cite{Lerman98}). In the $\Spinc$-case, existence of such a class is a necessary condition for  $\QSpinc_K(M, \mu)$ to be nonzero.
\begin{theorem}\label{thm kM h}
If there is no $(\kh) \in \cH_{\kk}$ such that $([\kk^M, \kk^M]) = ([\kh, \kh])$, then $\QSpinc_K(M, \mu)=0$.
\end{theorem}
This result will be proved in Subsection \ref{sec loc mult}. 
From now on, suppose $(\kh) \in \cH_{\kk}$ is given such that $([\kk^M, \kk^M]) = ([\kh, \kh])$.

To state Theorem \ref{thm [Q,R]=0}, recall that we chose a maximal torus $T<K$, with Lie algebra $\kt \subset \kk$, and a (closed) positive Weyl chamber  $\kt^*_+ \subset \kt^*$. Let $R$ be the set of roots of $(\kk_{\C}, \kt_{\C})$, and let $R^+$ be the set of positive roots with respect to $\kt^*_+$. Set
\[
\rhoK := \frac{1}{2}\sum_{\alpha \in R^+} \alpha.
\]

 Let $\cF$ be the set of relative interiors of faces of $\kt^*_+$. Then
\[
\kt^*_+ = \bigcup_{\sigma \in \cF} \sigma,
\]
a disjoint union. For $\sigma \in \cF$, let $\kk_{\sigma}$ be the infinitesimal stabiliser of a point in $\sigma$.
 Write
\begin{equation} \label{eq def Fh}
\cF(\kh) := \{\sigma \in \cF; (\kk_{\sigma}) = (\kh)\}.
\end{equation}
For such a $\sigma$, let $R_{\sigma}$ be the set of roots of $\bigl( (\kk_{\sigma})_{\C}, \kt_{\C}\bigr)$, and let $R_{\sigma}^+ := R_{\sigma} \cap R^+$. Set
\[
\rho_{\sigma} := \frac{1}{2}\sum_{\alpha \in R^+_{\sigma}} \alpha.
\]
Note that if $\sigma$ is the interior of $\kt^*_+$, then $\rho_{\sigma} = 0$.

Let $\Lambda_+ \subset i\kt^*$ be the set of dominant integral weights, and set $\Lambda_+^{\reg} := \Lambda_+ + \rhoK$. 
In the $\Spinc$-setting, it is natural to parametrise the irreducible representations by their infinitesimal characters, rather than by their highest weights. For $\lambda \in \Lambda_+^{\reg}$, let $\pi_{\lambda}$ be the irreducible representation
of $K$ with infinitesimal character $\lambda$, i.e.\ with highest weight $\lambda - \rhoK$. Then one has, for such $\lambda$,
\[
Q^{\Spinc}_K(K\cdot \lambda) = \pi_{\lambda},
\]
see Lemma 4.1 in \cite{Paradan15}.

For $\xi \in i\kk^*$, we write $M_{\xi} := M_{\xi/i}$. The $\Spinc$-quantisation $\QSpinc(M_{\xi})$ of such a reduced space, for the values $\xi$ of $\mu$ we will need, is defined in Section 5.3 of \cite{Paradan14}. This definition also applies to singular values of $\mu$. It involves realising $M_{\xi}$ as a reduced space $Y_{\eta}$ for the action by an Abelian group $A$ on a submanifold $Y$ of $M$. Then one uses the fact that, since $A$ is Abelian, the reduced space $Y_{\eta + \varepsilon}$ is a $\Spinc$-orbifold for generic $\varepsilon \in (\ka^Y)^{\perp}$. Hence its quantisation $\QSpinc(Y_{\eta + \varepsilon})$ is well-defined as the index of a $\Spinc$-Dirac operator, and turns out to be independent of small enough $\varepsilon$ (see Theorem 5.4 in \cite{Paradan14}). One sets $\QSpinc(M_{\xi}) := \QSpinc(Y_{\eta + \varepsilon})$, for generic and small enough $\varepsilon$.

Our main result is the following.
\begin{theorem}[Quantisation commutes with reduction] \label{thm [Q,R]=0}
Let $K$ be a compact, connected Lie group, and let $M$ be an even-dimensional, connected, 
$K$-equivariant $\Spinc$-manifold. Let $\mu$ be a $\Spinc$-moment map, and suppose it is proper. 
Suppose that there is an element $\xi \in \kk$ such that $\kh := \kk_{\xi}$ satisfies $([\kh, \kh]) = ([\kk^M, \kk^M])$.
 Then
\[
 Q^{\Spin^c}_K(M, \mu) = \bigoplus_{\lambda \in \Lambda_+^{\reg}} m_{\lambda} \pi_{\lambda},
\]
with $m_{\lambda} \in \Z$ given by
\begin{equation} \label{eq [Q,R]=0}
m_{\lambda} = \sum_{  \begin{array}{c} \vspace{-1.5mm} \scriptstyle{\sigma \in \cF(\kh) \text{ s.t.}}\\  \scriptstyle{\lambda - \rho_{\sigma} \in \sigma} \end{array}} Q^{\Spin^c}(M_{\lambda - \rho_{\sigma}}).
\end{equation}
Here  $\cF(\kh)$ is as in \eqref{eq def Fh}.
\end{theorem}
%
Theorem \ref{thm [Q,R]=0} will be proved in Section \ref{sec pf [Q,R]=0}.

If the generic stabiliser $(\kk^M)$ is \myemph{Abelian}, Theorem \ref{thm [Q,R]=0} simplifies considerably. 
This occurs in particular if $\mu$ has a regular value, since then $\kk^M = \{0\}$ (see Lemma 2.4 in \cite{Mathai14}).
\begin{corollary} \label{cor cpt Ab stab}
In the setting of Theorem \ref{thm [Q,R]=0}, if $(\kk^M)$ is Abelian, then for all $\lambda \in \Lambda_+^{\reg}$,
\[
m_{\lambda} =  Q^{\Spin^c}(M_{\lambda} ).
\]
\end{corollary}
\begin{proof}
If one takes $\kh = \kt$ in Theorem \ref{thm [Q,R]=0}, then $\cF(\kh)$ only contains the interior of $\kt^*_+$. Hence $\rho_{\sigma} = 0$, for the single element $\sigma \in \cF(\kh)$.
\end{proof}
An even more special case occurs  if $\rhoK$ is a regular value of $\mu$, and one only considers the multiplicity of the trivial representation.
\begin{corollary}
If $\rhoK$ is a regular value of $\mu$, then
\[
 Q^{\Spin^c}_K(M, \mu)^K = Q^{\Spin^c}(M_{\rhoK}).
\]
\end{corollary}
\begin{proof}
If $\mu$ has a regular value, then $\kk^M = \{0\}$. Since $\pi_{\rhoK}$ is the trivial representation, the claim follows from Corollary \ref{cor cpt Ab stab}.
\end{proof}

Theorem \ref{thm [Q,R]=0} is an analogue  in the more general $\Spinc$-setting of  
 the main result in \cite{Zhang14} for symplectic manifolds, also proved in \cite{Paradan11}. At the same time, it generalises the result in \cite{Paradan14} from compact manifolds to proper moment maps. This development fits into a long tradition of quantisation commutes with reduction results, which started with
 Guillemin and Sternberg's seminal paper \cite{Guillemin82} for compact K\"ahler manifolds. Results for compact symplectic manifolds were proved in
 \cite{Meinrenken98, Meinrenken99, Paradan01, Zhang98}. In \cite{HochsSong15b}, we use arguments like the ones in this paper to give a short proof of the symplectic version of Theorem \ref{thm [Q,R]=0}.

\begin{example}
In Theorem 4.1 in \cite{Paradan03}, Paradan proved a version of Theorem \ref{thm [Q,R]=0}, where the $\Spinc$-structure is associated to an almost complex structure, $\mu$ is a taming moment map in the symplectic sense, the stabilisers of the action are Abelian, and an additional assumption holds, Assumption 3.6 in \cite{Paradan03}. (Compared to Theorem \ref{thm [Q,R]=0}, Theorem 4.1 in \cite{Paradan03} includes a sign to account for the possibly different orientations induced by the almost complex structure and the symplectic form.) This Assumption 3.6 allows Paradan to prove the shifting trick, as in \eqref{eq m O 1}, in his context, via a homotopy argument (see Proposition 3.7 in \cite{Paradan03}). In the present paper, the required generalisation of the shifting trick follows from a multiplicativity result for the index, Theorem \ref{thm mult} below. The proof of this result in Section \ref{sec mult} is based on a different kind of homotopy argument, without the need for the assumption in \cite{Paradan03}.

Paradan applied his result to the interesting case of coadjoint orbits parametrising discrete series representations of semisimple Lie groups. 
Suppose $G$ is a connected, semisimple Lie group with discrete series, let $K<G$ be maximal compact, and let $T<K$ be a maximal torus. Let $\pi^{G}_{\nu}$ be the discrete series representation of $G$ with infinitesimal character $\nu \in i\kt^*$. Consider the coadjoint orbit $G\cdot \nu$. The positive root system for $(\kg_{\C}, \kt_{\C})$ of roots with positive inner products with $\nu$ determines a $G$-invariant complex structure on $G\cdot \nu \cong G/T$. For noncompact $G$, this is \emph{not} compatible with the Kirillov--Kostant--Souriau symplectic form $\omega_{\nu}$ on $G\cdot \nu$, so that one cannot apply the results in \cite{Zhang14, Paradan11} in the symplectic case to the action by $K$ on $G\cdot \nu$. Instead, one can use a $G$-equivariant $\Spinc$-structure whose determinant line bundle has first Chern class $[2\omega_{\nu}]$.

In Section 5.3 of \cite{Paradan03}, Paradan shows that Assumption 3.6 in that paper is satisfied in this example. Therefore, Theorem 4.1 in \cite{Paradan03} applies. We can now also directly apply Theorem \ref{thm [Q,R]=0}, without the need to check Assumption 3.6 in \cite{Paradan03}. In any case, the result is that
\[
\pi^G_{\nu}|_K = \bigoplus_{\lambda \in \Lambda_+^{\reg}} Q^{\Spinc}\bigl( (G\cdot \nu)_{\lambda}\bigr ) \pi_{\lambda}^K.
\]
Here
for clarity, we write $\pi^K_{\lambda} := \pi_{\lambda}$. See Proposition 5.2 in \cite{Paradan03}. The realisation of $\pi^G_{\nu}|_K$ as the $\Spinc$-quantisation of $G\cdot \nu$ is Theorem 5.1 in \cite{Paradan03}.
\end{example}

\subsection{Multiplicativity of the index} \label{sec thm mult}

As in the symplectic case \cite{Zhang14, Paradan11}, the main difficulty in proving Theorem \ref{thm [Q,R]=0} is to establish a generalisation of the shifting trick. In this subsection, we state a multiplicativity property, Theorem \ref{thm mult}, of the index  of Definition \ref{def index proper}. That will imply the version of the shifting trick we need in the present context, the equality \eqref{eq m O 1}. 

Let $N$ be a \myemph{compact}, connected, even-dimensional, $K$-equivariant $\Spinc$-manifold, with spinor bundle $\cS_N \to N$ and moment map $\mu_N: N \to \kk^*$. For clarity, we denote the spinor bundle $\cS$ on $M$ by $\cS_M$, and the moment map $\mu$ on $M$ by $\mu_M$ in this setting. Let $\hat \mu_M$ and $\hat \mu_N$ be the pullbacks of $\mu_M$ and $\mu_N$ to $M \times N$, along the two projection maps. Then
\[
\mu_{M \times N} := \hat \mu_M + \hat \mu_N: M\times N \to \kk^*
\]
is a $\Spinc$-moment map for the diagonal action by $K$ on $M \times N$, for the spinor bundle $\cS_{M \times N} := \cS_M \boxtimes \cS_N$. It is proper, because $N$ is compact. Compactness of $N$ also implies that the equivariant index 
\[
\ind_K(\cS_N) \quad \in R(K)
\]
of the $\Spinc$-Dirac operator on $N$ is well-defined by \eqref{eq index cpt} in the usual way, and equals the index of Definition \ref{def index Br}, for any taming map. 
Since $\ind_K(\cS_N)$ is finite-dimensional, the tensor product
\[
\indK(\cS_M, \mu_M) \otimes \ind_K(\cS_N) \quad \in \hat R(K) 
\]
is well-defined.

The index of Definition \ref{def index proper} is multiplicative in the following sense.
\begin{theorem}[Multiplicativity] \label{thm mult}
We have
\begin{equation} \label{eq mult}
\indK(\cS_{M\times N}, \mu_{M\times N}) = \indK(\cS_M, \mu_M) \otimes \ind_K(\cS_N) \quad \in \hat R(K).
\end{equation}
\end{theorem}
This result will be proved in Section \ref{sec mult}.

\begin{remark}
While superficially similar, Theorem \ref{thm mult} is considerably harder to prove than the multiplicativity property of the index in Theorem \ref{thm mult index}. This due to the term $\hat \mu_N$ in  $\mu_{M\times N} = \hat \mu_M + \hat \mu_N$. Theorem \ref{thm mult index} will be used in the proof of Theorem \ref{thm mult}.
\end{remark}


\section{Vanishing multiplicities} \label{sec vanishing}

We will give an analytic proof of Theorem \ref{thm vanishing}, by showing that certain deformed Dirac operators are positive on relevant spaces of sections. Two ingredients of the proof are a deformation of the moment map $\mu$, discussed in Subsection \ref{sec deform mu}, and an estimate for harmonic oscillator-type operators in Subsection \ref{sec harm osc}. 

In Subsections \ref{sec DT}--\ref{sec harm osc}, we will consider a $K$-equivariant $\Spinc$-manifold $M$, with spinor bundle $\cS \to M$, and determinant line bundle $L \to M$. We also fix a $K$-invariant Hermitian connection $\nabla^L$ on $L$, which induces a $\Spinc$-Dirac operator $D$ on $\cS$ as in Subsection \ref{sec Spinc Dirac}, and a moment map $\mu: M \to \kk^*$ as in Definition \ref{def mom map}. 

\subsection{The square of a deformed Dirac operator} \label{sec DT}
We will use an auxiliary real parameter $T \in \R$, and consider the deformed Dirac operator
\[
D_{T\mu} = D - \ii T c(\vf^{\mu}).
\]
 Using \eqref{eq Dirac local}, one can compute that
\begin{equation} \label{eq Bochner}
D_{T\mu}^2 = D^2 - \ii T\sum_{j=1}^{\dim M} c(e_j) c(\nabLC_{e_j}\vf^{\mu}) + 2\ii T \nabla^{\cS}_{\vf^{\mu}} + T^2 \|\vf^{\mu}\|^2,
\end{equation}
in terms of  a local orthonormal frame $\{e_1, \ldots, e_{\dim M}\}$  of $TM$. 

Fix $m \in Z_{\mu}$, and suppose $\alpha := \mu(m) \in \kt^*_+$. Then $m \in M^{\alpha} \cap \mu^{-1}(\alpha)$.
In the remainder of this section, we will identify $\kk^*$ with $\kk$ via the $\Ad(K)$-invariant inner product chosen earlier. In this way, we consider $\alpha$ as an element of $\kk$, and $\mu$ as a map from $M$ to $\kk$.  
Let $\alpha^M$ be the vector field defined as in \eqref{eq def XM}. The difference between the vector fields $\alpha^M$ and $v^{\mu}$ is that $\alpha^M$ is induced by the fixed element $\alpha \in \kk $, while for $m'\in M$, the tangent vector $v^{\mu}_{m'}$ is induced by the element $\mu(m')\in \kk $, depending on $m'$. These vector fields are equal at the point $m$, but not at other points in general. 

We will need to consider more general maps from $M$ to $\kk$ than just $\mu$ in the first estimate. Let $\varphi: M \to \kk$ be any equivariant map. Then we will write 
 $\cL^{\cS}_{\varphi}$ for the operator on $\Gamma^{\infty}(\cS)$ defined by
\begin{equation} \label{eq Lie mu}
(\cL^{\cS}_{\varphi}s)(m') = (\cL^{\cS}_{\varphi(m')}s)(m'),
\end{equation}
for all $m' \in M$ and $s \in \Gamma^{\infty}(\cS)$. 
\begin{lemma} \label{lem L nabla}
Suppose $\varphi(m) = \alpha$. Then for every $\varepsilon > 0$, there is a neighbourhood $U_m$ of $m$ in $M$, such that we have the inequality
\[
\left\|
 \ii(\nabla^{\cS}_{\vf^{\varphi}} - \cL^{\cS}_{\varphi})    -\|\alpha\|^{2} - \frac{\ii}{4} \Sj c(e_{j}) c(\nabla^{TM}_{e_{j}}\alpha^{M})  \right\|\leq \varepsilon.
\]
when restricted to smooth sections of $\cS$ with compact supports inside $U_m$.
\end{lemma}
\begin{proof}
Because of the local decomposition \eqref{eq decomp S}, 
\[
\ii \bigl(\nabla^{\cS}_{\alpha^M} - \cL^{\cS}_{\alpha}  \bigr)  = \mu_{\alpha} + \frac{\ii}{4}  \sum c(e_{j}) c(\nabla^{TM}_{e_{j}}\alpha^M).
\]

Let $\varepsilon > 0$ be given. 
At $m$, we have
\[
\mu_{\alpha}(m) = \|\alpha\|^2.
\]
So in a small enough neighbourhood $U_m$ of $m$, we have
\[
\left| \mu_{\alpha} - \|\alpha\|^2 \right| \leq \varepsilon/2.
\]
Furthermore, the vector bundle endomorphism
$
\nabla^{\cS}_{X^M} - \cL^{\cS}_{X} 
$
depends continuously on $X \in \kk$. So by choosing $U_m$ small enough, we can ensure that
\[
\left\|
\ii ( \nabla^{\cS}_{\vf^{\varphi}} - \cL^{\cS}_{\varphi} ) - \ii ( \nabla^{\cS}_{\alpha^M} - \cL^{\cS}_{\alpha} ) \right\| \leq \varepsilon/2.
\]
on $U_m$. The claim follows.
\end{proof}


\subsection{Deforming moment maps} \label{sec deform mu}

In the expression \eqref{eq Bochner} for the squared operator $D_{T\mu}^2$, the operator 
\[
\sum_{j=1}^{\dim M} c(e_j) c(\nabLC_{e_j}\vf^{\mu})
\] 
occurs. To compare this operator to the operator
\[
\Sj c(e_{j}) c(\nabla^{TM}_{e_{j}}\alpha^{M})
\]
in Lemma \ref{lem L nabla},
 we use a local deformation of the moment map $\mu$. This is a new  addition to Tian and Zhang's analytic approach in the symplectic case \cite{Zhang98}. 

For the point $\alpha \in \kk $, one can choose a $K_{\alpha}$-invariant open subset $Z \subset \kk_{\alpha} $ containing $\alpha$, such that the map
\[
\begin{split}
K \times_{K_{\alpha}} Z &\to K\cdot Z\\
 [k, \xi] &\mapsto k\cdot \xi,
\end{split}
\]
for $k \in K$ and $\xi \in Z$,
is a diffeomorphism.

One can show that $T_m \mu(T_mM) + \kk_{\alpha} = \kk$, so that, for $Z$ small enough, $Y := \mu^{-1}(Z)$ is a smooth submanifold of $M$. Since $\mu$ is equivariant, $Y$ is $K_{\alpha}$-invariant. Furthermore, 
 we have an equivariant diffeomorphism
\begin{equation}
\label{equivariance}
 K \times_{K_{\alpha}} Y \to W := K\cdot Y
\end{equation}
onto an $K$-invariant open neighbourhood $W$ of $m$.

Let $T^{\alpha} < K_{\alpha}$ be the torus generated by $\alpha$, and let $\kt^{\alpha}$ be its Lie algebra. Let $\kh$ be the orthogonal complement to $\kt^{\alpha}$ in $\kk_{\alpha}$, and let $\mu_{\kt^{\alpha}}$ and $\mu_{\kh}$ be the projections of $\mu|_Y$ to $\kt^{\alpha} $ and $\kh$, respectively.
For $t \specialin ]0,1]$, define the map $\mu^t_Y: Y \to \kk_{\alpha}$ by
\[
\mu^t_Y := \mu_{\kt^{\alpha}} + t \mu_{\kh}.
\]
Because $\kt^{\alpha}$ is in the centre of $\kk_{\alpha}$, the decomposition $\kk_{\alpha} = \kt^{\alpha} \oplus \kh$ is $K_{\alpha}$-invariant. So both components $\mu_{\kt^{\alpha}}$ and $\mu_{\kh}$ of $\mu|_Y$ are $K_{\alpha}$-equivariant. Therefore, $\mu^t_Y$
extends $K$-equivariantly to a map 
\[
\mu^t: W \to \kk.
\] 
In particular, $\mu^1 = \mu|_W$. We denote by $\vf^{\mu^t}$ the vector field on $W$ induced 
by $\mu^t$. 

\begin{lemma} \label{lem Z mu t}
We can choose the set $Z$, and hence the sets $Y$ and $W$, such that the vector field $\vf^{\mu^t}$ has the same zeroes 
  for all $t \specialin ]0,1]$.
\end{lemma}
\begin{proof}
By $K$-invariance of $\vf^{\mu^t}$, it is enough to show that the set of zeroes of $\vf^{\mu^t}|_Y$ is independent of $t \specialin ]0,1]$, for $Y$ small enough.

To see that this is true, first note that for all $m' \in Y \cap M^{\alpha}$, we have
\[
\vf^{\mu_{\kt^{\alpha}}}_{m'} = 0,
\]
so 
\[
\vf^{\mu^t}|_{Y \cap M^{\alpha}} = t \vf^{\mu_{\kh}}|_{Y \cap M^{\alpha}},
\]
which has the same zeroes for all nonzero $t$. 

Next, consider the set $Y \cap (M \setminus M^{\alpha})$. By choosing $Z$, and hence $Y$ by properness of $\mu$, to be relatively compact, we can ensure that the action by $K_{\alpha}$ on $Y$ only has finitely many infinitesimal stabiliser types. Each of these stabiliser types defines a closed subset of $\kk_{\alpha}$, which does not contain $\alpha$. Therefore, there is a $\delta > 0$ such that for all $X \in \kk_{\alpha}$ with $\|X-\alpha\| \leq \delta$ and all $m' \in Y \cap (M \setminus M^{\alpha})$, we have $X^M_{m'} \not= 0$. Choose $Z$ so that all elements of $Z$ lie within a distance $\delta$ of $\alpha$. Then for all $m' \in Y \cap (M \setminus M^{\alpha})$ and any $t \specialin ]0,1]$,
\begin{multline*}
\| \mu^t(m') - \alpha  \|^2 = \|\mu_{\kt^{\alpha}}(m') - \alpha\|^2 + t^2 \|\mu_{\kh}(m')\|^2\\
 \leq  \|\mu_{\kt^{\alpha}}(m') - \alpha\|^2 + \|\mu_{\kh}(m')\|^2 = \| \mu(m') - \alpha  \|^2 \leq \delta^2. 
\end{multline*}
Hence $\vf^{\mu^t}_{m'} \not=0$ for any $t \specialin ]0,1]$. 
\end{proof}
Because the vector field $\vf^{\mu^t}$ has the same set of zeroes for all $t \specialin ]0,1]$, Theorem \ref{thm htp invar} implies that for all such $t$, if $W \cap Z_{\mu}$ is compact,
\begin{equation} \label{eq index mu t}
\indK(\cS|_W, \mu^t) = \indK(\cS|_W, \mu|_W).
\end{equation}

\subsection{An estimate for deformed moment maps}

The reason for introducing the deformation $\mu^t$ of $\mu|_W$ 
in Subsection \ref{sec deform mu}
is the following estimate.
\begin{lemma} \label{lem est mu alpha}
For any $\varepsilon > 0$, there is a $\delta \specialin ]0,1]$ such that 
for all $t \specialin ]0, \delta]$ 
and any orthonormal basis $\{f_1, \ldots, f_{\dim Y}\}$ of $T_mY$, one has at $m$,
\begin{equation} \label{eq est V mu alpha}
\left\| \sum_{j=1}^{\dim M} c(e_j) c(\nabLC_{e_j}\vf^{\mu^t}) -  \sum_{j=1}^{\dim Y} c(f_j) c(\nabLC_{f_j}\alpha^M)  \right\| \leq \varepsilon.
\end{equation}
\end{lemma}
\begin{proof}
We use the decomposition
\[
T_mM \cong T_mY \oplus \kk_{\alpha}^{\perp}. 
\]
Note that for all $X \in \kk$ and for all $t$, torsion-freeness of the Levi--Civita connection implies that
\[
\nabLC_{X^M}\vf^{\mu^t} = \nabLC_{\vf^{\mu^t}}X^M - \cL^{TM}_X \vf^{\mu^t}.
\]
Because $\vf^{\mu^t}$ is zero at $m$, and $K$-equivariant, the right hand side of this equality vanishes at $m$. So
\begin{equation} \label{eq nab V X mu t zero}
\bigl(\nabLC_{X^M}\vf^{\mu^t} \bigr)_m = 0.
\end{equation}
Let $\{f_1, \ldots, f_{\dim M}\}$ be an orthonormal basis of $T_mM$ such that $f_j \in T_mY$ if $j \leq \dim Y$, and $f_j \in \kk_{\alpha}^{\perp}$ if $j>\dim Y$. Then \eqref{eq nab V X mu t zero} implies that for all $j > \dim {Y}$
\[
\bigl(\nabLC_{f_j}\vf^{\mu^t} \bigr)_m = 0.
\]
Since the operator 
\[
\sum_{j=1}^{\dim M} c(e_j) c(\nabLC_{e_j}\vf^{\mu^t}) 
\]
is independent of the frame $\{e_1,\ldots, e_{\dim M}\}$, we find that at $m$, it equals
\begin{equation} \label{eq op Clifford}
\sum_{j=1}^{\dim Y} c(f_j) c\bigl(  (\nabLC_{f_j}\vf^{\mu^t})_m \bigr). 
\end{equation}

Next, let $\{X_1, \ldots, X_{\dim {\kk_{\alpha}}}\}$ be an orthonormal basis of $\kk_{\alpha}$, such that $X_k \in \kt^{\alpha}$ if $k \leq \dim {\kt^{\alpha}}$, and $X_k \in \kh$ if $k > \dim {\kt^{\alpha}}$. For every $k$, set
\[
\mu^t_k := \mu^t_{X_k}|_Y \quad \in C^{\infty}(Y). 
\]
Since $\mu^t(Y) \subset \kk_{\alpha}$, we have
\[
\mu^t|_Y = \sum_{k = 1}^{\dim {\kk_{\alpha}}} \mu^t_k X_k.
\]
Also note that $Y$ is $K_{\alpha}$-invariant, so $X_k^M|_Y = X_k^Y$ for all $k$. Therefore, for all $j$,
\[
\left(\nabLC_{f_j} \vf^{\mu^t} \right)_m = \sum_{k = 1}^{\dim {\kk_{\alpha}}} 
\left( \mu^t_k(m) (\nabLC_{f_j}X_k^Y)_m +  f_j(\mu^t_k)(m) (X_k^Y)_m \right).
\]
We find that \eqref{eq op Clifford} equals
\begin{equation} \label{eq op Clifford 2}
 \sum_{j=1}^{\dim Y}  c(f_j) c\bigl( ( \nabLC_{f_j}\alpha^Y)_m  \bigr) +
 \sum_{k = 1}^{\dim {\kk_{\alpha}}}  c(\grad_m \mu^t_k) c\bigl((X_k^Y)_m \bigr).
\end{equation}

To bound the second term in \eqref{eq op Clifford 2}, note that for $k \leq \dim {\kt_{\alpha}}$, we have $(X_k^Y)_{m} = 0$, and for $k > \dim {\kt^{\alpha}}$, 
\[
\mu^t_k(m) = t\mu_{X_k}(m). 
\]
Let $\varepsilon > 0$. If we choose $\delta \specialin ]0,1]$ such that
\[
\delta \sum_{k = \dim {\kt^{\alpha}}+1}^{\dim {\kk_{\alpha}}} \|\grad_m \mu_{X_k}\| \cdot  \|(X_k^Y)_m\| \leq \varepsilon,
\]
then we see that \eqref{eq est V mu alpha} holds for all $t \specialin ]0, \delta]$, at the point $m$. 
\end{proof}

\begin{remark}
If $K$ is a torus, then $K_{\alpha} = K$, so that $Y = W$ is an open neighbourhood of $m$ in $M$. This simplifies the arguments in this subsection.
\end{remark}


\subsection{Vector fields and Lie derivatives}

At several points, we will use a convenient local expression for vector fields around their zeroes.
\begin{lemma} \label{lem loc V}
Let $\vf \in \cX(M)$ be a vector field. Let $m \in M$ such that $\vf_m = 0$. Then there are $a_1, \ldots, a_{\dim_M} \geq 0$, there is an orthogonal automorphism $J$ of $TM$, defined near $m$, 
and there is a local orthonormal frame $\{ e_1, \ldots, e_{\dim M} \}$ of $TM$ near $m$, such that in the normal coordinates $y = (y_1, \ldots, y_{\dim M})$ associated to the corresponding basis of $T_mM$,
\begin{equation}
\label{expression}
\vf_y = \Sj a_j y_j J e_j + \cO(\|y\|^2).
\end{equation}
\end{lemma}
\begin{proof}
Let $\{f_1, \ldots, f_{\dim M}\}$ be any local orthonormal frame of $TM$ near $m$. Let $x = (x_1, \ldots, x_{\dim M})^T$ (viewed as a column vector) be the normal coordinates associated to the basis $\{(f_1)_m, \ldots, (f_{\dim M})_m\}$ of $T_mM$. Since $\vf_m = 0$, there are $b_{jk} \in \R$ such that
\[
\vf_x = \sum_{j, k = 1}^{\dim M} b_{jk} x_j f_k + \cO(\|x\|^2).
\]
Let $B$ be the matrix with elements $(b_{jk})_{j, k = 1}^{\dim M}$. Write
\[
B = \sqrt{B^TB}U,
\]
for $U \in \OO(\dim M)$, and
\[
\sqrt{B^TB} = WD_a W^T,
\]
for $W \in \SO(\dim M)$, and for a diagonal matrix 
\[
D_a := \diag(a_1, \ldots, a_{\dim M}),
\]
with each of the entries $a_j$ nonnegative. For any vector $u \in \R^n$, and any set of vectors $w = (w_1, \ldots, w_n) \in \R^n$, we write
\[
uw := \sum_{j=1}^n u_j w_j  \quad \in \R^n.
\]
Then
\[
\begin{split}
\sum_{j, k = 1}^{\dim M} b_{jk} x_j f_k &= x^T B f \\
	&= x^T WD_aW^TUf \\
	&= (W^TU x)^T (W^T U W) D_a W^T U f.
\end{split}
\]
Set $e_j := W^TU f_j$. The normal coordinates associated to the basis 
\[
\{(e_1)_m, \ldots, (e_{\dim M})_m\}
\] 
of $T_mM$ are 
$y := WU^T x$.
Hence one obtains (\ref{expression}) by defining $J := W^T U W$.
\end{proof}

If $V$ is any real vector space, and $B \in \End(V)$, then we write
\[
|B| := \sqrt{B^TB}.
\]
\begin{lemma} \label{lem tr abs Lie}
In the setting of Lemma \ref{lem loc V}, the Lie derivative
\[
\cL_{\vf}^{T_mM} \in \End(T_mM)
\]
satisfies
\[
\tr|\cL_{\vf}^{T_mM}| = \sum_{j=1}^{\dim M} a_j.
\]
\end{lemma}
\begin{proof}
In the situation of Lemma \ref{lem loc V}, one can compute that the Lie derivative operator $\cL_{\vf}^{T_mM}$ on $T_mM$ equals
\[
\cL_{\vf}^{T_mM} = -\Sj a_j J (e_j)_m \otimes (e_j^*)_m, 
\]
where $e_j^*$ is the one-from dual to $e_j$.
So with respect to the basis $\{ (e_1)_m, \ldots, (e_{\dim M})_m \}$ of $T_mM$, the map $\cL^{T_mM}_\vf$ has matrix
\[
\mat \cL^{T_mM}_\vf = - \mat (J) D_a,
\]
where, as before, $D_a$ is the diagonal matrix with entries $\{a_1, \ldots, a_{\dim M}\}$. Hence $\mat |\cL^{T_mM}_\vf| = D_a$.
\end{proof}

\subsection{A local estimate}

Lemma \ref{lem tr abs Lie} in particular yields an expression for the trace $\tr|\cL^{T_mM}_{\alpha}|$ in terms of the local expression in Lemma \ref{lem loc V} of the vector field $\alpha^M$. This allows us to prove the following estimate.

%
%
%
%

\begin{lemma} \label{lem est 1}
At $m$, we have
\[
\left \| \sum_{j=1}^{\dim M} c(e_j)c(\nabLC_{e_j} \alpha^M) \right\|\leq \tr|\cL_{\alpha}^{T_mM}|.
\]
\end{lemma}
\begin{proof}
Write
\[
\alpha^M_m = \Sj a_j y_j Je_j + \cO(\|y\|^2)
\]
as in Lemma \ref{lem loc V}. Since the Christoffel symbols of $\nabla^{TM}$ in the coordinates $y$ vanish at $m$, we have for all $j$,
\[
\bigl(\nabla^{TM}_{e_j}\alpha^M\bigr)_m = a_j (Je_j)_m.
\]
Thus, Lemma \ref{lem tr abs Lie} implies that at $m$, 
\[
\left\| \sum_{j=1}^{\dim M} c(e_j)c(\nabLC_{e_j} \alpha^M) \right\| = \left\|\Sj a_j c(e_j)_m c(Je_j)_m \right\| \leq \Sj a_j = \tr|\cL_{\alpha}^{T_mM}|.
\]
\end{proof}

\begin{lemma} \label{lem loc est}
For all $\varepsilon > 0$, there is a neighbourhood $U_m$ of $m$, and a constant  $\delta > 0$, such that, when restricted to smooth sections with compact supports inside $U_m$, we have for all $t \specialin ]0, \delta]$,
\[
D_{T\mu^t}^2 - 2T\ii \cL^{\cS}_{\mu^t} \geq  T\left( 2\|\alpha\|^2 - \frac{1}{2}\tr|\cL_{\alpha}^{T_mM}| - \varepsilon \right) + D^2  + T^2 \|\vf^{\mu^t}\|^2.
\]
\end{lemma}
\begin{proof}
Let $\varepsilon > 0$ be given.
By \eqref{eq Bochner} and Lemma \ref{lem L nabla}, applied with $\varphi = \mu^t$, we can choose $U_m$ so small that on $U_m$,
\begin{multline*}
D_{T\mu^t}^2 - 2T\ii \cL^{\cS}_{\mu^t} \geq  D^2  + T^2 \|\vf^{\mu^t}\|^2 
+2T \|\alpha\|^{2} \\- \ii T \sum_{j=1}^{\dim M} c(e_j) c(\nabLC_{e_j}\vf^{\mu^t}) + \frac{\ii T}{2} \Sj c(e_{j}) c(\nabla^{TM}_{e_{j}}\alpha^{M})
 -\varepsilon T/2.
\end{multline*}
By Lemma \ref{lem est mu alpha}, there is a $\delta \specialin ]0,1]$ such that for all $t \specialin ]0, \delta]$, at $m$,
\[
-\ii \sum_{j=1}^{\dim M} c(e_j) c(\nabLC_{e_j}\vf^{\mu^t}) \geq
	-\ii \sum_{j=1}^{\dim Y} c(f_j) c(\nabLC_{f_j}\alpha^M) - \varepsilon/2,
\]
for any orthonormal basis $\{f_1, \ldots, f_{\dim Y}\}$ of $T_mY$. Extending this to an orthonormal basis $\{f_1, \ldots, f_{\dim M}\}$ of $T_mM$, such that 
\[
f_j \in \kk_{\alpha}^{\perp} \cong \kk/\kk_{\kk_{\alpha}}, 
\] 
for $\dim Y < j \leq \dim M$,
and noting that the operator $\Sj c(e_{j}) c(\nabla^{TM}_{e_{j}}\alpha^{M})$ is independent of the local orthonormal frame $\{e_1, \ldots, e_{\dim M}\}$, we find that, at $m$, 
\begin{equation}
\begin{aligned}
&- \ii  \sum_{j=1}^{\dim M} c(e_j) c(\nabLC_{e_j}\vf^{\mu^t}) + \frac{\ii }{2} \Sj c(e_{j}) c(\nabla^{TM}_{e_{j}}\alpha^{M})\\
&\geq - \frac{\ii}{2} \sum_{j=1}^{\dim Y} c(f_j) c(\nabLC_{f_j}\alpha^M) + \frac{\ii}{2} \sum_{j = \dim Y+1}^{\dim M} c(f_{j}) c(\nabla^{TM}_{f_{j}}\alpha^{M}) - \varepsilon/2.
\end{aligned}
\end{equation}
By applying Lemma \ref{lem est 1} to $Y$ and $K/K_{\alpha}$, we find that the latter expression is at least equal to
\[
-\frac{1}{2}\tr|\cL_{\alpha}^{T_mY}|- \frac{1}{2}\tr|\cL_{\alpha}^{\kk/\kk_{\alpha}}| - \varepsilon/2 = -\frac{1}{2}\tr|\cL_{\alpha}^{T_mM}|- \varepsilon/2. 
\]
This completes the proof. 
\end{proof}


\subsection{Harmonic oscillator} \label{sec harm osc}

Let us define an operator
\begin{equation} \label{eq def KT}
K_{T, t} = D^2  - T\tr|\cL_{\vf^{\mu^t}}^{T_mM}|+ T^2 \|\vf^{\mu^t}\|^2.
\end{equation}
It can be bounded below as follows.
\begin{lemma}\label{lem bound KT}
For any $\varepsilon > 0$, there is a neighbourhood $U_m$ of $m$, and a constant $C>0$, such that for all $t \specialin ]0, 1]$, 
on sections supported in $U_m$,
\[
K_{T, t} \geq -C -T\varepsilon.
\]
\end{lemma}
\begin{proof}
By Lemma \ref{lem loc V}, we can write
\[
\vf^{\mu^t} = \sum_{j=1}^{\dim M} h_{j}(t, \relbar) Je_{j},
\]
near $m$,
where, for all $j$,
\[
h_{j}(t, y) = a_{j}(t)y_j +  r_j(t, y),
\]
for smooth functions $a_j$ and $r_j$, such that $a_j$ is nonnegative, and
\[
r_j(t, y) = \cO(\|y\|^2).
\]
We will write $h_j^t := h_j(t, \relbar)$ and $r_j^t := r_j(t, \relbar)$. Then
\[
\| \vf^{\mu^t} \|^{2} = \sum_{j=1}^{\dim M} (h_j^t)^{2}, \quad \text{and} \quad \tr|\cL_{\vf^{\mu^t}}^{T_mM}| = \Sj a_j(t).
\]

Analogously to (2.17) in \cite{Zhang98}, consider the nonnegative operator
\begin{equation} \label{eq def DelT}
\Delta_{T, t} := \sum_{j=1}^{\dim {M}} \left(  \bigl(\nabla^{\cS}_{e_j}\bigr)^* + T \cdot h_j^t \right)\left( \nabla^{\cS}_{e_j} + T\cdot h_j^t \right).
\end{equation}
By a straightforward computation, we have
\begin{equation} \label{eq DelT}
\Delta_{T, t} = \Delta -T \tr|\cL_{\vf^{\mu^t}}^{T_mM}| + T^2 \|\vf^{\mu^t}\|^2 + 
	T\sum_{j=1}^{\dim M}  \left(  \nabla^{\cS}_{e_j} +  \bigl(\nabla^{\cS}_{e_j}\bigr)^*   \right) h_j^t -T e_j(r_j^t).
\end{equation}
Here $\Delta := \Sj   \bigl(\nabla^{\cS}_{e_j}\bigr)^*  \nabla^{\cS}_{e_j}$ is the Bochner Laplacian. By the Bochner formula (see e.g.\ Theorem D.12 in \cite{Lawson89}), the difference $\Delta - D^2$ is a vector bundle endomorphism of $\cS$. So is the operator $ \nabla^{\cS}_{e_j} + \bigl(\nabla^{\cS}_{e_j}\bigr)^*$. 

Let $\varepsilon > 0$ be given. Note that $h_j(t, m) = 0$ for all $t$ and $j$, and that $h_j(t, y)$ depends smoothly on $t$, and extends smoothly to $t \in \R$.
Therefore, we can choose $U_m$ so small that for all $t$ in the compact interval $[0,1]$, we have on $U_m$,
\[
\left\| \Sj  \left( \nabla^{\cS}_{e_j} + \bigl(\nabla^{\cS}_{e_j}\bigr)^*  \right) h_j^t  \right\| \leq \varepsilon/2. 
\]
Similarly, $e_j(r_j^t)(y) = \cO(\|y\|)$, and this function is smooth in $t$. This allows us to choose $U_m$ small enough so that for all $t \in [0,1]$, we have
\[
|e_j(r_j^t) | \leq \varepsilon/2.
\]
With $U_m$ chosen in this way, one has the desired lower bound for $K_{T, t}$.
\end{proof}

In addition, we have the following estimate for the middle term $\tr|\cL_{\vf^{\mu^t}}^{T_mM}|$ in (\ref{eq def KT}). 

\begin{lemma} \label{lem tr La Vmu}
For all $\varepsilon > 0$, there is a $\delta > 0$, such that for all $t \specialin ]0, \delta]$, we have
\[
 \tr|\cL_{\vf^{\mu^t}}^{T_mM}| \geq \tr|\cL_{\alpha}^{T_mY}| - \varepsilon
\] 
\end{lemma}
\begin{proof}
For all $X \in \kk$ and for all $t$, $K$-invariance of the vector field $\vf^{\mu^t}$ implies that
\[
\cL_{\vf^{\mu^t}} X^M = -\cL_X \vf^{\mu^t} = 0. 
\]
Thus, $ \tr|\cL_{\vf^{\mu^t}}^{T_mM}| =  \tr|\cL_{\vf^{\mu^t}}^{T_mY}|$ and it is enough to prove the inequality on the slice $Y$. 

Let $\varepsilon > 0$ be given.
Recall that
\[
{\mu^t}|_Y = \mu_{\kt^{\alpha}} + t\mu_{\kh}.
\]
Accordingly, 
\[
\vf^{\mu^t}|_Y = \vf^{\mu_{\kt^{\alpha}}}  + t\vf^{\mu_{\kh}}.
\]
Since $ \tr|\cL_{\vf^{\mu^t}}^{T_mY}|$ depends continuously on $t$, there is a $\delta \specialin ]0,1]$ such that for all $t \specialin ]0, \delta]$,
\[
 \tr|\cL_{\vf^{\mu^t}}^{T_mY}| \geq  \tr|\cL_{\vf^{\mu_{\kt^{\alpha}}}}^{T_mY}| - \varepsilon. 
\]

It remains to compare $\tr|\cL_{\vf^{\mu_{\kt^{\alpha}}}}^{T_mY}|$ with $\tr|\cL_{\alpha}^{T_mY}|$. Let us write
\[
\mu_{\kt^{\alpha}}(y) = \alpha + \sum_{j=1}^{\dim \kt^{\alpha}} y_j X_j + \cO(\|y\|^2),
\]
for  coordinates $y$ on $Y$ near $m$, and $X_1, \ldots, X_{\dim Y} \in \kt^{\alpha}$. 
Since $(X_j^Y)_m = 0$ for such $X_j$, one has that
\[
(X_j^Y)_y = \sum_{k, l = 1}^{\dim Y} c_{j}^{kl}y_k e_l + \cO(\|y\|^2),
\]
for certain numbers $c_j^{kl}$, and for a local orthonormal frame $\{e_1, \ldots, e_{\dim Y}\}$ of $TY$. It follows that 
\[
\vf^{\mu_{\kt^{\alpha}}}_y - \alpha^Y_y = \cO(\|y\|^2).
\]
Therefore, Lemma \ref{lem tr abs Lie} implies that
\[
\tr|\cL_{\vf^{\mu_{\kt^{\alpha}}}}^{T_mY}| = \tr|\cL_{\alpha}^{T_mY}|.
\]
Hence the claim follows. 
\end{proof}

In \cite{Paradan14}, a function $d$ on $Z_{\mu}$ plays an important role. We will now see this function appear in our estimates as well. It is defined by
\begin{equation} \label{def func d}
d(m') := \|\mu(m')\|^2 +  \frac{1}{4} \tr |\cL^{T_{m'}M}_{\mu(m')}| - \frac{1}{2}\tr |\ad(\mu(m'))|, 
\end{equation}
for $m' \in Z_{\mu}$. It is locally constant by Lemma 4.16 in \cite{Paradan14}. Furthermore, we note for later use that there is a constant $C_K > 0$, 
independent of $M$ or $\mu$, such that for all $m' \in Z_{\mu}$,
\begin{equation} \label{eq est d}
d(m') \geq \|\mu(m')\|^2 - C_K \|\mu(m')\|.
\end{equation}
The precise value of $C_K$ is not important for our arguments, but to be specific we can take $C_K := 2 \|\rhoK\|$ (see e.g.\ the comment above Lemma 3.10 in \cite{Paradan14}).

\begin{proposition} \label{prop loc est d}
For all $\varepsilon > 0$, there is a neighbourhood $U_m$ of $m$, and  constants  $\delta > 0$ and $C>0$, such that, when restricted to smooth sections with compact supports inside $U_m$, we have for all $t \specialin ]0, \delta]$,
\[
D_{T\mu^t}^2 - 2T\ii \cL^{\cS}_{\mu^t} \geq  T(2d(m) - \varepsilon) - C. 
\]
\end{proposition}
\begin{proof}
By combining Lemmas  \ref{lem bound KT} and \ref{lem tr La Vmu}, we find that for $U_m$ small enough, and for a $\delta \specialin ]0, 1]$ and a $C>0$, we have for all $t \specialin ]0, \delta]$, on sections supported in $U_m$,
\[
D^2  + T^2 \|\vf^{\mu^t}\|^2 \geq T\cdot (\tr|\cL_{\alpha}^{T_mY}| - \varepsilon/2) - C.
\]
So by Lemma \ref{lem loc est}, we can choose $U_m$ and $\delta \specialin ]0,1]$ so small that for all $t \specialin ]0, \delta]$, on sections supported in $U_m$,
\[
D_{T\mu^t}^2 - 2T\ii \cL^{\cS}_{\mu^t} \geq  T\left( 2\|\alpha\|^2 + \tr |\cL^{T_mY}_{\alpha}| - \frac{1}{2}\tr|\cL^{T_mM}_{\alpha}|- \varepsilon \right) -C.
\]
Remembering the fact that $T_mM = T_mY \oplus \kk/\kk_{\alpha}$ and $\ad(\alpha) = 0$ on $\kk_{\alpha}$, one sees that
\begin{equation} \label{eq est tr Lie}
\tr |\cL^{T_mY}_{\alpha}| - \frac{1}{2}\tr|\cL^{T_mM}_{\alpha}| = \frac{1}{2} \tr |\cL^{T_mM}_{\alpha}| - \tr |\ad(\alpha)|.
\end{equation}
\end{proof}


\subsection{Proof of Theorem \ref{thm vanishing}} \label{sec pf vanishing d}

Fix $\pi \in \hat K$. Let $B_{\pi} > 0$ be such that
the infinitesimal representation of $\kk$ associated to $\pi$ satisfies
\[
\|\pi(X)\| \leq B_{\pi} \|X\|,
\]
 for all $X \in \kk$.
Set
\[
C_{\pi} := B_{\pi} + 2 \|\rhoK\|.
\]
We will prove Theorem \ref{thm vanishing} by showing that this value of $C_{\pi}$ has the desired property.

Let $U$ and $\mu$ be as in Theorem \ref{thm vanishing}. In particular, suppose that $\|\mu(m)\| > C_{\pi}$ for all $m \in U$. Let $F$ be a connected component of $Z_{\mu}$, and let $\alpha$ be the value of $\mu$ on $F$. Then $\|\alpha\| > C_{\pi}$. Choose $\varepsilon > 0$ such that
\begin{equation} \label{eq def eta}
\eta := 2\|\alpha\|\bigl(\|\alpha\| - C_{\pi}\bigr) - \varepsilon ( 2B_{\pi} + 1 ) > 0.
\end{equation}
By Lemma \ref{lem decomp Z mu} and properness of $\mu$, the set $F$ is compact.
Let $W$ be a $K$-invariant, relatively compact neighbourhood of $F$ on which the deformed moment map $\mu^t$ of Subsection \ref{sec deform mu} is defined, and such that $Z_{\mu^t}$ is independent of $t \specialin ]0, 1]$ (see Lemma \ref{lem Z mu t}).

Since $F$ is compact, 
Proposition \ref{prop loc est d} allows us to find an open cover $\{V_1, \ldots, V_n\}$ of $F$ such that for all $j$, there are $\delta_j \specialin ]0,1]$, $C_j > 0$ and $m_j \in F \cap V_j$, such that for all $t \specialin ]0, \delta_j]$, we have on smooth sections of $\cS$ supported inside $V_j$,
\[
D_{T\mu^t}^2   \geq  T\bigl(2d(m_j) + 2\ii \cL^{\cS}_{\mu^t}- \varepsilon \bigr) - C_j. 
\]
The function $d$ is constant on $F$, so $d(m_j) = d(m)$ for any fixed $m \in F$. Set $\delta := \min_j \delta_j$ and $C := \max_j C_j$. Then for all $j$,  we have the estimate
\begin{equation} \label{eq est DTmut}
D_{T\mu^t}^2   \geq  T\bigl(2d(m) + 2\ii \cL^{\cS}_{\mu^t}- \varepsilon \bigr) - C,
\end{equation}
for $t \specialin ]0, \delta]$, on sections supported in $V_j$.

By shrinking $W$ if necessary, we may assume that $\overline{W} \subset \bigcup_{j=1}^n V_j$. As on pp.\ 115--117 of \cite{Bismut91}, and on p.\ 243 of \cite{Zhang98}, choose functions $\varphi_j$ supported in $V_j$, for each $j$, such that $\sum_{j=1}^n \varphi_j^2 = 1$ on $W$, and which allow us to conclude that \eqref{eq est DTmut} holds on the space of all smooth sections of $\cS$ supported in $W$.

Next, let $\Gamma^{\infty}_c(\cS|_W)_{\pi}$ be the $\pi$-isotypical component of $\Gamma^{\infty}_c(\cS|_W)$. On this subspace, Lie derivatives are bounded.
\begin{lemma} \label{lem bound L mu t}
For all $\varepsilon > 0$, the $K$-invariant neighbourhood $W$ of $F$ can be chosen so that the operator $\cL^{\cS}_{\mu^t}$ is bounded on $\Gamma^{\infty}_c(\cS|_W)_{\pi}$ with respect to the $L^2$-norm, with norm at most
\[
\left\|\cL^{\cS}_{\mu^t}|_{\Gamma^{\infty}_c(\cS|_W)_{\pi}}  \right\| \leq B_{\pi}(\|\alpha\| + \varepsilon).
\]
In addition, the neighbourhood $W$ with this property can be chosen independently of $\pi$.
\end{lemma}
\begin{proof}
Consider the unitary isomorphism
\[
\Phi: \bigl( L^2(K) \otimes L^2(\cS|_Y) \bigr)^{K_{\alpha}} \xrightarrow{\cong} L^2(\cS|_W),
\]
defined by
\[
\bigl(\Phi(\psi \otimes s) \bigr)(k\cdot y) := \psi(k) k\cdot (s(y)),
\]
fro $\psi \in L^2(K)$, $s \in L^2(\cS|_Y)$, $k \in K$ and $y \in Y$.
It is $K$-equivariant with respect to the left regular representation of $K$ in $L^2(K)$. So the inverse image of $\Gamma^{\infty}_c(\cS|_W)_{\pi}$ lies inside $\bigl( L^2(K)_{\pi} \otimes L^2(\cS|_Y) \bigr)^{K_{\alpha}}$. (Where $L^2(K)_{\pi}$ is the sum of $\dim \pi$ copies of $\pi$.)

Also, for any $K$-equivariant map $\varphi: W \to \kk$, one can check that, on smooth sections,
\begin{equation} \label{eq Lie Phi}
\cL^{\cS|_W}_{\varphi} \circ \Phi = \Phi \circ \bigl(1 \otimes \cL^{\cS|_Y}_{\varphi|_Y} \bigr),
\end{equation}
where the Lie derivative operators are defined as in \eqref{eq Lie mu}.

The constant map $Y \to \kk$ with value $\alpha$ is $K_{\alpha}$-equivariant, and hence extends to a $K$-equivariant map $\tilde \alpha: W \to \kk$. By \eqref{eq Lie Phi}, we have, on smooth sections,
\[
\cL^{\cS|_W}_{\tilde \alpha} \circ \Phi = \Phi \circ \bigl(1 \otimes \cL^{\cS|_Y}_{\alpha} \bigr).
\]
And since $\alpha \in \kk_{\alpha}$, we have on the smooth part of $\bigl( L^2(K) \otimes L^2(\cS|_Y) \bigr)^{K_{\alpha}}$,
\[
1 \otimes \cL^{\cS|_Y}_{\alpha} = -  \cL_{\alpha} \otimes 1.
\]
Therefore, the operator $\cL^{\cS|_W}_{\tilde \alpha}$ is bounded on $\Gamma^{\infty}_c(\cS|_W)_{\pi}$, with norm at most $B_{\pi} \|\alpha\|$.

Furthermore,
\[
\cL^{\cS|_W}_{\mu^t - \tilde \alpha} \circ \Phi = \Phi \circ \bigl(1 \otimes \cL^{\cS|_Y}_{\mu^t|_Y - \alpha} \bigr).
\]
Let $\{X_1, \ldots, X_{\dim \kk_{\alpha}}\}$ be an orthonormal basis of $\kk_{\alpha}$. Write
\[
\mu^t_j := \mu^t_{X_j}|_{Y} \quad \text{and} \quad \alpha_j := (\alpha, X_j).
\]
Then, since $\mu^t(Y) \subset \kk_{\alpha}$,
\[
\cL^{\cS|_Y}_{\mu^t|_Y - \alpha} = \sum_{j=1}^{\dim \kk_{\alpha}} (\mu^t_j - \alpha_j) \cL^{\cS|_Y}_{X_j}.
\]
As before, we have $1 \otimes \cL^{\cS|_Y}_{X_j} = - \cL_{X_j} \otimes 1$ on the smooth part of $\bigl( L^2(K) \otimes L^2(\cS|_Y) \bigr)^{K_{\alpha}}$. So the difference $\cL^{\cS|_W}_{\mu^t} - \cL^{\cS|_W}_{\tilde \alpha}$ is bounded on $\Gamma^{\infty}_c(\cS|_W)_{\pi}$, with norm at most
\[
\left\| \left. \left(\cL^{\cS|_W}_{\mu^t} - \cL^{\cS|_W}_{\tilde \alpha} \right)\right|_{\Gamma^{\infty}_c(\cS|_W)_{\pi}} \right\| \leq B_{\pi} \sum_{j=1}^{\dim \kk_{\alpha}} \|\mu^t_j - \alpha_j \|_{\infty},
\]
where $\|\cdot\|_{\infty}$ denotes the supremum norm. 

Let $\varepsilon > 0$. Because $\mu^t(Y \cap F) = \{\alpha\}$, we can choose the set $Y$, and hence $W$, so small that 
\[
\sum_{j=1}^{\dim \kk_{\alpha}} \|\mu^t_j|_Y - \alpha_j \|_{\infty} \leq \varepsilon.
\]
This neighbourhood $W$ has the desired property.
\end{proof}

As noted below \eqref{eq est d}, 
we have
\[
d(m) \geq \|\alpha\|^2 - 2 \|\alpha\| \|\rhoK\|.
\]
Combining this with Lemma \ref{lem bound L mu t} and the fact that \eqref{eq est DTmut} holds on $W$,  we conclude that, on sections in $\Gamma^{\infty}_c(\cS)_{\pi}$ supported in $W$,
\[
\begin{split}
D_{T\mu^t}^2  &\geq  T\bigl(2\|\alpha\|^2 - 4 \|\rhoK\| \|\alpha\| - 2 B_{\pi} (\|\alpha\| + \varepsilon) - \varepsilon
\bigr) - C \\
	&> T  \eta - C, 
\end{split}
\]
with $\eta > 0$ as in \eqref{eq def eta}. Therefore,  if $T > C/ \eta$, then $D_{T\mu^{t}}^2 > 0$ on the space of such sections. Because of \eqref{eq index mu t}, this implies that
\[
 \bigl[ \indK(\cS|_W, \mu|_W) : \pi \bigr] = \bigl[ \indK(\cS|_W, T\mu^t) : \pi \bigr] = 0. 
\]
By summing over all connected components of $Z_{\mu}$, we find that
\[
\bigl[ \indK(\cS, \mu) : \pi \bigr] = 0, 
\]
so Theorem \ref{thm vanishing} is true.

\begin{remark}
With Definition \ref{def index proper}, we immediately see that Theorem \ref{thm vanishing} generalises from taming moment maps to proper moment maps.
\end{remark}

A final comment, which we will use in the proof of Theorem \ref{thm [Q,R]=0}, is that
 if $\pi$ is the trivial representation, the above reasoning leads to a more precise vanishing result. 
 Indeed, on $K$-invariant sections, the operator $\cL^{\cS}_{\mu^t}$ is zero, so
the inequality (\ref{eq est DTmut}) becomes
\[
D_{T\mu^{t}}^2\geq  T\left(2 d(m) - \varepsilon \right) -C.
\]
Hence one gets the following result.
 \begin{proposition} \label{prop vanishing d}
 For every even-dimensional $K$-equivariant $\Spinc$-manifold $M$, with spinor bundle $\cS$ and proper moment map $\mu$, one has
\[
 \indK(\cS, \mu)^K = 0,
 \]
if the function $d$ is strictly positive.
 \end{proposition}
 
\begin{remark}
Another approach to proving Theorem \ref{thm vanishing} would be  to note that Proposition 4.17 in \cite{Paradan14} generalises to the proper moment map case, because its proof in \cite{Paradan14} is based on local computations near connected components of $Z_{\mu}$. This would yield Proposition \ref{prop vanishing d}.

Via the estimate \eqref{eq est d}, this implies the case of Theorem \ref{thm vanishing} where $\pi$ is the trivial representation. This is enough to prove Theorem \ref{thm mult} for the $K$-invariant parts of both sides of \eqref{eq mult}, as in Section \ref{sec mult}. That in turn can be used to deduce Theorem \ref{thm vanishing} from the case of the trivial representation, via a shifting trick. 

The analytic proof of Theorem \ref{thm vanishing} in this section makes this paper more self-contained than the approach sketched above would. In addition, it illustrates the power of analytic localisation techniques, and highlights the key role of the function $d$ on $Z_{\mu}$ that is used. This function is also of central importance in \cite{Paradan14}, and it is interesting to see it emerge here in a very different way.
\end{remark}


\section{Multiplicativity} \label{sec mult}

The proof of Theorem \ref{thm mult} is based on the invariance of Braverman's index under homotopies of taming maps, Theorem \ref{thm htp invar}. An important condition in the definition of such a homotopy is that the map connecting two given taming maps is taming itself. In our arguments, we can make sure this condition is satisfied by replacing given taming maps by proper ones.

\subsection{Making taming maps proper} \label{sec mu proper}

A key ingredient of our proof of Theorem \ref{thm mult} is the possibly surprising fact that a taming moment map can always be replaced by a proper one, without changing the resulting index. This is in fact possible for any taming map.

Let $U$ be a connected, complete, even-dimensional  manifold, and suppose $g$ is a $K$-invariant Riemannian metric on $M$.
Let $\varphi: U \to \kk$ be a taming map, i.e.\ $Z_{\varphi}$ is compact. 
%
\begin{proposition} \label{prop mu proper}
Let $V\subset U$ be a $K$-invariant, relatively compact neighbourhood of $Z_{\varphi}$. 
Then there is a  taming map $\tilde \varphi: M \to \kk$ with
the following properties:
 \begin{itemize}
 \item $\tilde \varphi$ is proper;
 \item $\tilde \varphi|_V = \varphi|_V$;
 \item $\|\tilde \varphi\| \geq \|\varphi\|$;
 \item the vector fields $\vf^{\varphi}$ and $\vf^{\tilde \varphi}$ have the same set of zeroes;
 \item the function $\| \vf^{\tilde \varphi}\|$ on $U$ is proper.
 \end{itemize}
 In addition, $\varphi = \mu$ is a $\Spinc$-moment map for a $K$-equivariant $\Spinc$-structure on $U$, then $\tilde \varphi = \tilde \mu$ can be chosen to be a $\Spinc$-moment map for the same $\Spinc$-structure, but a different connection on the determinant line bundle.
\end{proposition}
(Note that the first part of this proposition is not specific to the $\Spinc$-setting.)

In the setting of Proposition \ref{prop mu proper}, if $\cE \to U$ is a $K$-equivariant Clifford module, then
Proposition \ref{prop excision} implies that
\[
\indK(\cE, \varphi) = \indK(\cE, \tilde \varphi).
\]

To prove Proposition \ref{prop mu proper}, we consider a
 nonnegative, $K$-invariant function $\theta \in C^{\infty}(U)^K$, and the map $\psi: U\to \kk$ defined by
 \begin{equation} \label{eq def psi}
 (\psi, X) := g(v^{\varphi}, X^U)
 \end{equation}
 for $X\in \kk$. Here the vector field $X^U$ is defined as in \eqref{eq def XM}. 
 We define the map $\varphi^\theta: U \to \kk$  by
\[
\varphi^{\theta} := \varphi + \theta \psi.
\]
If $\varphi = \mu$ is a $\Spinc$-moment map associated to a connection $\nabla^L$ on the determinant line bundle $L \to U$ of a given $\Spinc$-structure, then $\mu^{\theta}$ is the  $\Spinc$-moment map associated to the connection
\[
\nabla^L - 2 i \theta \alpha,
\]
on $L$, where $\alpha \in \Omega^1(U)^K$ is the one-form dual to $\vf^{\varphi}$. 
(This deformation of $\nabla^L$ is closely related to the deformation of Dirac operators as in Definition \ref{def deformed Dirac}, see Remark \ref{rem deform conn}.) 
We will prove Proposition \ref{prop mu proper} by showing that the map $\varphi^{\theta}$ has the desired properties for a well-chosen function $\theta$.

Let $\{X_1, \ldots, X_{\dim {\kk}}\}$ be an orthonormal  basis of $\kk$. Then 
\[
\varphi^{\theta} =  {\varphi} + {\theta} \sum_{j=1}^{\dim {\kk}} g(\vf^{\varphi}, X_j^U) X_j,
\]
so that
\begin{equation} \label{eq V mu f}
\vf^{\varphi^{\theta}} = \vf^{\varphi} + {\theta} \sum_{j=1}^{\dim {\kk}} g(\vf^{\varphi}, X_j^U) X_j^U.
\end{equation}
\begin{lemma} \label{lem V mu V mu f}
Let $m \in U$. Then $\vf^{\varphi^{\theta}}_m  = 0$ if and only if $\vf^{\varphi}_m = 0$.
\end{lemma}
\begin{proof}
Let $m \in U$.
First, suppose $\vf^{\varphi}_m = 0$. Then 
$\varphi^{\theta}(m) = \varphi(m)$, so
$
\vf^{\varphi^{\theta}}_m  = \vf^{\varphi}_m  = 0.
$

To prove the converse implication, note that \eqref{eq V mu f} implies that
\begin{equation} \label{eq inner prod vfs}
g(\vf^{\varphi}, \vf^{\varphi^{\theta}}) = \|\vf^{\varphi}\|^2 + {\theta} \sum_{j=1}^{\dim {\kk}} g(\vf^{\varphi}, X_j^U)^2.
\end{equation}
Suppose $\vf^{\varphi^{\theta}}_m  = 0$. Since the second term on the right hand side of \eqref{eq inner prod vfs} is nonnegative, this implies that $\|\vf^{\varphi}_m\|^2 = 0$.
\end{proof}

In addition, the norm of $\varphi^{\theta}$ is as least as great as the norm of $\|\varphi\|$.
\begin{lemma} \label{lem norms mu f mu}
One has
\begin{equation} \label{eq norms mu f mu}
\|\varphi^{\theta}\| \geq \|\varphi\|.
\end{equation}
\end{lemma}
\begin{proof}
By definition of $\varphi^{\theta}$,
we have for all $m \in U$,
\[
\bigl(\varphi^{\theta}(m), \varphi(m) \bigr) = 
 \|\varphi(m)\|^2 + {\theta}\|\vf^{\varphi}_m\|^2 \geq \|\varphi(m)\|^2.
\]
By the Cauchy--Schwartz inequality, this implies that
\[
\|\varphi^{\theta}(m)\| \|\varphi(m)\| \geq \|\varphi(m)\|^2. 
\]
Hence \eqref{eq norms mu f mu} follows outside $\varphi^{-1}(0)$. And if $\varphi(m) = 0$, then $\vf^{\varphi}_m = 0$, so $\varphi^{\theta}(m) = \varphi(m)$.
\end{proof}

Proposition \ref{prop mu proper} follows from Lemmas \ref{lem V mu V mu f} and \ref{lem norms mu f mu},  
if we choose
$\tilde \varphi := \varphi^{\theta}$
with ${\theta}$ as in the following lemma.
\begin{lemma}
Let $V \subset U$ be a relatively compact, $K$-invariant neighbourhood of $Z_{\varphi}$. Then the function ${\theta}$ can be chosen such that ${\theta}|_V \equiv 0$, and the function  $\|\vf^{\varphi^{\theta}}\|$ and the  map $\varphi^{\theta}$  are proper.
\end{lemma}
\begin{proof}
Fix a point $m_0 \in V$. For any $m \in U$, let $\delta(m)$ be the Riemannian distance from $m$ to $m_0$. 
Write
\[
B_r := \{x \in U; \delta(m) \leq r\}.
\] 
Choose $r_0 > 0$ such that $V \subset B_{r_0}$.

Since $\vf^{\varphi}$ is tangent to orbits, and does not vanish outside $V$, the map $\psi$ does not vanish outside $V$. Choose ${\theta} \in C^{\infty}(U)^K$ such that
\begin{itemize}
\item ${\theta}|_{B_{r_0}} \equiv 0$;
\item for all $m \in U\setminus B_{r_0 + 1}$,
\[
{\theta}(m) \geq \max\left( \frac{\delta(m)}{\|\psi(m)\|^2}, \frac{\|\varphi(m)\| + \delta(m)}{\|\psi(m)\|} \right).
\]
\end{itemize}
For this choice of $\theta$, we have
\[
\theta \|\psi\|^2 \geq \delta,
\]
outside $B_{r_0 + 1}$,
so $\theta \|\psi\|^2$ is a proper function. Since
\[
\|\vf^{\varphi^{\theta}}\|^2 = \|\vf^{\varphi} + \theta v^{\psi}\|^2 \geq 2\theta g(\vf^{\psi}, \vf^{\varphi}) = 2\theta \|\psi\|^2,
\]
the function $\|\vf^{\varphi^{\theta}}\|$ is proper as well.

Next, 
note that 
\[
\|\varphi^{\theta}\| \geq \theta \Bigl\| \sum_{j=1}^{\dim {\kk}} g\bigl(\vf^{\varphi}, X_j^U\bigr) X_j \Bigr\| - \|\varphi\|.
\]
Since 
\[
 \Bigl\| \sum_{j=1}^{\dim {\kk}} g\bigl(\vf^{\varphi}, X_j^U \bigr) X_j \Bigr\| = \|\psi(m)\|,
\]
we have
for all $r \geq r_0 + 1$ and all $m \in U \setminus B_r$,
\[
\|\varphi^{\theta}(m)\| \geq {\theta}\|\psi(m)\| - \|\varphi(m)\| \geq \delta(m) \geq r.
\]
So the inverse image under $\varphi^{\theta}$ of the ball in $\kk$ of radius $r$ is contained in $B_r$. Because $U$ is complete, this inverse image is therefore compact. 
\end{proof}


\subsection{A first localisation} \label{sec mult invar}

Consider the setting of Subsection \ref{sec thm mult}. In the proof of Theorem \ref{thm mult}, we will use invariance of Braverman's index under homotopies of taming maps, as in Theorem \ref{thm htp invar}. To construct a suitable homotopy, we decompose, as in Lemma \ref{lem decomp Z mu},
\begin{align}
Z_{\mu_M} &= \bigcup_{\alpha \in \Gamma_M} K\cdot  \left( M^{\alpha} \cap \mu_M^{-1}(\alpha)\right); \label{eq decomp Z mu M} \\
Z_{\mu_{M\times N}} &= \bigcup_{\beta \in \Gamma_{M\times N}} K\cdot  \left( (M\times N)^{\beta} \cap \mu_{M\times N}^{-1}(\beta)\right) \label{eq decomp Z mu M N},
\end{align}
for discrete subsets $\Gamma_M$ and $\Gamma_{M\times N}$ of $\kt^*_+$.

Let $\pi \in \hat K$ be an irreducible representation of $K$. For such a $\pi$, let $C_{\pi}$ be the constant in Theorem \ref{thm vanishing}. Set
\[
C^N_{\pi} := \max\bigl\{C_{\pi'}; \text{$\pi' \in \hat K$ is equal to $\pi$ or occurs in $\ind_K(\cS_N)^*\otimes \pi$}\bigr\}.
\]
(Note that $\ind_K(\cS_N)$ is finite-dimensional because $N$ is compact, so this maximum is well-defined.)
Let $C_N>0$ be such that $\|\mu_N(n)\| < C_N$ for all points $n$ in the compact manifold $N$. Consider the set
\[
U_M := \{m \in M; \|\mu_M(m)\| < C^N_{\pi} + 3C_N + 1 + \varepsilon\},
\]
for an $\varepsilon > 0$ such that
\begin{itemize}
\item 
for all $\alpha \in \Gamma_M$, $\|\alpha\| \not= C^N_{\pi} + 3C_N + \varepsilon$;
\item $C^N_{\pi} + 3C_N + \varepsilon$ is a regular value of $\|\mu_M\|$, and $\partial U_M$ is a smooth hypersurface in $M$.
\end{itemize}
Furthermore, we will use the set
\[
U^{M \times N} := \{(m, n) \in M \times N; \|\mu_{M\times N}(m, n) \| < C^N_{\pi} + 2C_N + 1 + \varepsilon\},
\]
for an $\varepsilon > 0$ such that
for all $\beta \in \Gamma_{M\times N}$, $\|\beta\| \not= C^N_{\pi} + 2C_N + 1 + \varepsilon$.

By properness of $\mu_M$ and compactness of $N$, the sets $U_M$ and $U^{M \times N}$ are relatively compact. Furthermore, we have
\[
\begin{split}
Z_{\mu_M} \cap \partial U_M& = \emptyset; \\
Z_{\mu_{M\times N}} \cap \partial U^{M \times N}& = \emptyset.
\end{split}
\]
So  $Z_{\mu_M} \cap U_M$ and $Z_{\mu_{M\times N}} \cap U^{M \times N}$ are compact. By Theorem \ref{thm vanishing} and Proposition \ref{prop additivity}, we have
\begin{equation} \label{eq loc mu M N}
\bigl[ \indK(\cS_{M\times N}, \mu_{M \times N}) : \pi \bigr] = \bigl[ \indK(\cS_{M\times N}|_{U^{M \times N}}, \mu_{M \times N}|_{U^{M \times N}}) : \pi \bigr]. 
\end{equation}
Here we used the fact that  $\| \mu_{M\times N}(m, n)\| > C_{\pi}$ for $(m, n) \in (M\times N)\setminus U^{M\times N}$.

Using Theorem \ref{thm mult index}, and then applying Theorem \ref{thm vanishing} and Proposition \ref{prop additivity} on $M$, we also find that  
\begin{equation} \label{eq loc hat mu M}
\begin{split}
\bigl[ \indK(\cS_{M\times N}, \hat \mu_{M}) : \pi \bigr] &= \bigl[ \indK(\cS_{M},  \mu_{M}) \otimes \ind_K(\cS_N): \pi \bigr] \\
&= \sum_{\pi' \in \hat K} [\ind_K(\cS_N)^* \otimes \pi: \pi'] \cdot \bigl[ \indK(\cS_{M},  \mu_{M}): \pi' \bigr] \\
&= \sum_{\pi' \in \hat K} [\ind_K(\cS_N)^* \otimes \pi: \pi'] \cdot \bigl[ \indK(\cS_{M}|_{U_M},  \mu_{M}|_{U_M}): \pi' \bigr] \\
&= \bigl[ \indK(\cS_{M\times N}|_{U_M \times N}, \hat \mu_{M}|_{U_M \times N}) : \pi \bigr]. 
\end{split}
\end{equation}
Here we used the fact that for all $\pi'$ occurring in $\ind_K(\cS_N)^* \otimes \pi$ and all  $m \in M\setminus U_M$, we have $\| \mu_{M}(m)\| > C_{\pi'}$.

\subsection{A homotopy of taming maps}

Consider the set
\[
V := \{(m, n) \in M \times N; \|\mu_{M\times N}(m, n) \| < C^N_{\pi} + 2C_N + \varepsilon\},
\]
for an $\varepsilon  \specialin ]0, 1[$ such that
for all $\beta \in \Gamma_{M\times N}$, $\|\beta\| \not= C^N_{\pi} + 2C_N + \varepsilon$.
Its closure $\overline{V}$ is contained in $(U_M \times N) \cap U^{M \times N}$. Indeed, it obviously lies inside $U^{M \times N}$, while for all $(m, n) \in \overline{V}$,
\[
\|\mu_M(m)\| \leq \|\mu_{M\times N} (m, n)\|+ C_N \leq C^N_{\pi} + 3C_N + 1,
\]
so $m \in U_M$. In particular, the closure of the projection $V_M$ of $V$ to $M$ is contained in $U_M$. 

Since $\partial U_M$ is a smooth hypersurface, we can make $U_M$ complete by rescaling the Riemannian metric as in Subsection \ref{sec index noncomplete}. Hence Proposition \ref{prop mu proper} applies to $\mu_M|_{U_M}$.
Let $\tilde \mu_{U_M}: U_M \to \kk^*$ be the resulting proper moment map,
chosen such that $\tilde \mu_{U_M}|_{V_M} = \mu_M|_{V_M}$. (Note that $V_M$ is not necessarily a neighbourhood of $Z_{\mu} \cap U_M$, but it is contained in such a neighbourhood, which is enough.) Define the map
\[
\tilde \mu_{U_M \times N}: U_M \times N \to \kk^*
\]
by
\[
\tilde \mu_{U_M \times N}(m, n) = \tilde \mu_{U_M}(m) + \mu_N(n),
\]
for $m \in U_M$ and $n \in N$. Then $\tilde \mu_{U_M \times N}$ is proper, because $\tilde \mu_{U_M}$ is, and $N$ is compact. 

We will use a homotopy argument to prove the following result.
\begin{proposition} \label{prop htp mult}
The map $\tilde \mu_{U_M \times N}$ is taming, and we have
\begin{equation} \label{eq htp mult}
\indK\bigl(\cS_{M\times N}|_{U_M \times N}, \tilde \mu_{U_M \times N}\bigr) = \indK\bigl(\cS_{M\times N}|_{U_M \times N}, \hat \mu_M|_{U_M \times N}\bigr). 
\end{equation}
\end{proposition}

First, note that  by the comment below Proposition \ref{prop mu proper}, we have
\begin{equation} \label{eq index mu mu tilde}
\indK\bigl(\cS_{M\times N}|_{U_M \times N}, \hat \mu_M|_{U_M \times N}\bigr) = \indK\bigl(\cS_{M\times N}|_{U_M \times N}, \widehat{ \widetilde{\mu}}_{U_M}\bigr).
\end{equation}
Therefore, to prove Proposition \ref{prop htp mult}, it is enough to show that the left hand side of \eqref{eq htp mult} equals the right hand side of \eqref{eq index mu mu tilde}. 
To prove that equality, we will construct a homotopy of taming maps between $ \tilde \mu_{U_M \times N}$ and $\widehat{ \widetilde{\mu}}_{U_M}$. Then Proposition \ref{prop htp mult} follows from Theorem \ref{thm htp invar}.

Let $\lambda \in C^{\infty}(\R)$ be a function with values in $[0,1]$, such that
\[
\lambda(t) = \left\{\begin{array}{ll} 
0 & \text{if $t \leq 1/3$;} \\
1 & \text{if $t \geq 2/3$.}
\end{array} \right.
\]
Set $W := U_M \times N \times [0,1]$. Define the map $\varphi: W \to \kk^*$ by
\[
\varphi(m, n, t) := \tilde \mu_{U_M}(m) + \lambda(t) \mu_N(n),
\]
for $m \in U_M$, $n \in N$ and $t \in [0,1]$. We will show that $Z_{\varphi}$ is compact, so that $\varphi$ defines a homotopy of taming maps between $\tilde \mu_{U_M \times N}$ and $\widehat{ \widetilde{\mu}}_{U_M}$. This also implies that $\tilde \mu_{U_M \times N}$ is taming. Therefore, the following lemma implies Proposition \ref{prop htp mult}.
\begin{lemma} \label{lem Z phi cpt}
The set $Z_{\varphi} \subset W$ where $\vf^{\varphi}$ vanishes is compact.
\end{lemma}
\begin{proof}
We first claim that
the set
\[
Z := \bigl\{m\in U_M; \vf^{\tilde \mu_{U_M}}_m + s (\mu_N(n))^M_m = 0 \text{ for some $(n, s) \in N\times [0,1]$} \bigr\}
\]
is relatively compact in $U_M$. (Here $(\mu_N(n))^M$ is the vector field on $M$ induced by $\mu_N(n) \in \kk$, as in \eqref{eq def XM}.)
Indeed, 
the function $(m, n)\mapsto \|(\mu_N(n))^M_m\|$ on $M\times N$ is bounded on the relatively compact set $U_M \times N$. Let $C>0$ be an upper bound. If $m\in U_M$, $n\in N$ and $s \in [0,1]$, and 
\[
\vf^{\tilde \mu_{U_M}}_m + s (\mu_N(n))^M_m = 0, 
\]
then
\[
\| \vf^{\tilde \mu_{U_M}}_m \| =  s \|(\mu_N(n))^M_m \| \leq C.
\]
By the fifth point in Proposition \ref{prop mu proper}, the function $\| \vf^{\tilde \mu_{U_M}} \|$ on $U_M$ is proper, so $Z$ is relatively compact.

Now
let $(m, n, t) \in Z_{\varphi}$. Then
\[
0 = \vf^{\varphi}_{(m, n, t)} = \bigl(v^{\tilde \mu_{U_M}}_m + \lambda(t)(\mu_N(n))^M_m, (\tilde \mu_{U_M}(m))^N_n + \vf^{\mu_N}_n \bigr) \quad \in T_mM \times T_nN.
\]
The vanishing of the first component implies that $m$ lies in the relatively compact set $Z$. Since $Z_{\varphi}$ is closed, the claim follows.
\end{proof}

\subsection{Proof of Theorem \ref{thm mult}}

After Proposition \ref{prop htp mult}, the next step in the proof of Theorem \ref{thm mult} is the following application of Theorem \ref{thm vanishing}.
\begin{lemma} \label{lem vanish mult}
We have
\[
\bigl[ \indK(\cS_{M\times N}|_{U^{M \times N}}, \mu_{M \times N}|_{U^{M \times N}}) : \pi \bigr] = \bigl[ \indK(\cS_{M\times N}|_{U_M \times N}, \tilde \mu_{U_M \times N}) : \pi \bigr].
\]
\end{lemma}
\begin{proof}
By definition of the set $V$, we have for all $(m, n) \in (U_M \times N) \setminus V$,
\[
\|\mu_{M\times N}(m, n)\| \geq C^N_{\pi} + 2C_N + \varepsilon > C_{\pi}.
\]
And if $(m, n) \in U^{M \times N} \setminus V$, then the third point in Proposition \ref{prop mu proper} implies that
\[
\begin{split}
\|\tilde \mu_{U_M \times N}(m, n)\| &= \|  \tilde \mu_{U_M}(m) + \mu_N(n)\| \\
	&\geq \|  \tilde \mu_{U_M}(m)\| - C_N \\
	&\geq  \|   \mu_{U_M}(m)\| - C_N \\
	&\geq  \|   \mu_{U_M}(m)+ \mu_N(n)\| - 2C_N \\
	&> C_{\pi}.
\end{split}
\]
Furthermore, the choice of the number $\varepsilon$ in the definition of the set $V$ implies that the vector field $\vf^{\mu_{M\times N}}$ does not vanish on $\partial V$. 
Because  $V\subset V_M \times N$, and $\tilde \mu_{U_M}|_{V_M} = \mu_M|_{V_M}$, we have
\begin{equation}\label{eq mu mu tilde}
\mu_{M \times N}|_{V} = \tilde \mu_{U_M \times N}|_V.
\end{equation}
Therefore, 
 the vector fields $\vf^{\mu_{M\times N}}$ and $\vf^{\tilde \mu_{U_M\times N}}$ coincide on $V$. So the latter vector field does not vanish on $\partial V$ either. Finally, note that both $\mu_{M \times N}|_{U^{M \times N}}$ and $\tilde \mu_{U_M \times N}$ are taming moment maps. 

By the preceding arguments, Theorem \ref{thm vanishing} and Proposition \ref{prop additivity} imply that
\[
\bigl[ \indK(\cS_{M\times N}|_{U^{M \times N}}, \mu_{M \times N}|_{U^{M \times N}}) : \pi \bigr] = 
	\bigl[ \indK(\cS_{M\times N}|_{V}, \mu_{M \times N}|_{V}) : \pi \bigr],
\]
and
\[
\bigl[ \indK(\cS_{M\times N}|_{U_M \times N}, \tilde \mu_{U_M \times N}) : \pi \bigr] =  \bigl[ \indK(\cS_{M\times N}|_{V}, \tilde \mu_{U_M \times N}|_V) : \pi \bigr].
\]
Applying  \eqref{eq mu mu tilde} to the right hand sides of these equalities, we conclude that the claim holds.
\end{proof}

The final ingredient of the proof of Theorem \ref{thm mult} is the multiplicativity property of Braverman's index in Theorem \ref{thm mult index}.  Consecutively applying \eqref{eq loc mu M N}, Lemma \ref{lem vanish mult}, Proposition \ref{prop htp mult},  \eqref{eq loc hat mu M} and Theorem \ref{thm mult index}, we find that 
\[
\begin{split}
\bigl[ \indK(\cS_{M\times N}, \mu_{M \times N}) : \pi \bigr] &= \bigl[ \indK(\cS_{M\times N}|_{U^{M \times N}}, \mu_{M \times N}|_{U^{M \times N}}) : \pi \bigr] \\
&= \bigl[ \indK(\cS_{M\times N}|_{U_M \times N}, \tilde \mu_{U_M \times N}) : \pi \bigr] \\
&=\bigl[\indK\bigl(\cS_{M\times N}|_{U_M \times N}, \hat \mu_M|_{U_M \times N}\bigr):\pi \bigr] \\
&= \bigl[ \indK(\cS_{M\times N}, \hat \mu_{M}) : \pi \bigr] \\
&= \bigl[\indK(\cS_M, \mu_M) \otimes \ind_K(\cS_N):\pi \bigr].
\end{split}
\]
So Theorem \ref{thm mult} is true.


\section{Quantisation commutes with reduction} \label{sec pf [Q,R]=0}

As noted at the start of Subsection \ref{sec thm mult}, proving Theorem \ref{thm mult} was the main part of the work to prove Theorem \ref{thm [Q,R]=0}. In the remainder of the argument, the function $d$, defined in \eqref{def func d}, plays an important role. We start by discussing some relevant properties of this function in Subsection \ref{sec prop d}. Then, in Subsection \ref{sec loc mult}, we combine these with Theorem \ref{thm mult} to obtain an expression for the multiplicities in Theorem \ref{thm [Q,R]=0}, localised near a compact set. Then we are in the same situation as in \cite{Paradan14}. This will allow us to prove Theorem \ref{thm kM h}.
 In Subsection \ref{sec decomp mult}, we indicate how to apply, and generalise where necessary, the arguments in \cite{Paradan14} needed to finish the proof of Theorem \ref{thm [Q,R]=0}.

Throughout this section, we consider the setting of Theorem \ref{thm [Q,R]=0}.

\subsection{Properties of the function $d$} \label{sec prop d}

Let $d$ be the function on $Z_{\mu}$ defined in \eqref{def func d}. Localising $K$-invariant parts of indices to neighbourhoods of $d^{-1}(0)$ will be an important step in the proof of Theorem \ref{thm [Q,R]=0}. The properties of $d$ discussed in this subsection will be used in that localisation.

First of all, the estimate \eqref{eq est d} for $d$ implies that
for all $C_1 > 0$, there is a constant $C_2 >0$ such that for any $K$-equivariant $\Spinc$-manifold with moment map $\mu$, one has for all
$m \in Z_{\mu}$, 
\begin{equation} \label{eq est d mu}
\|\mu(m)\| \geq C_2 \quad \Rightarrow \quad d(m) \geq C_1.
\end{equation}
As a consequence, properness of $\mu$ implies that the function $d$ is proper as well. We will write $Z_{\mu}^{<0}$, $Z_{\mu}^{=0}$ and $Z_{\mu}^{>0}$ for the subsets of $Z_{\mu}$ where the function $d$ is negative, zero and positive, respectively. Then $Z_{\mu}^{=0}$ is compact by properness of $d$. By \eqref{eq est d}, the function $d$ is bounded below, so that the set $Z_{\mu}^{<0}$ is compact as well (though we will not use this).
In addition, we have the following generalisation of Lemma 4.16 in \cite{Paradan14} to our setting.
\begin{lemma} \label{lem 4.16}
There is a neighbourhood of $Z_{\mu}^{=0}$ disjoint from $Z_{\mu}^{<0}$ and $Z_{\mu}^{>0}$. 
(Hence $Z_{\mu}^{<0}$, $Z_{\mu}^{=0}$ and $Z_{\mu}^{>0}$ are all unions of connected components of $Z_{\mu}$.)
\end{lemma}
\begin{proof}
Let $C_2$ be the constant in \eqref{eq est d mu}, 
 with $C_1 = 1$. 
Then $\|\mu \| \leq C_2$ on $d^{-1}\bigl( [-1, 1] \bigr)$. The arguments in the proof of Lemma 4.16 in \cite{Paradan14} therefore show that $d$ takes finitely many values in $[-1, 1]$. Hence there is an $\varepsilon  \specialin ]0,1[$ such that  $d^{-1}\bigl(]-\varepsilon, \varepsilon[ \bigr)$ is is the desired neighbourhood of $Z_{\mu}^{=0}$.
\end{proof}

By Proposition \ref{prop vanishing d}, neighbourhoods of $Z_{\mu}^{>0}$ will not contribute to invariant parts of indices. In addition, the set where $d$ is negative is empty for the manifolds we will consider.  Let $\cO \cong K/T$ be a regular, admissible coadjoint orbit of $K$. Here admissibility means that $\cO$ has a $K$-equivariant $\Spinc$-structure, for which the inclusion map $\mu^{\cO}: \cO \hookrightarrow \kk^*$ is a moment map. Let $d_{\cO}$ be the function on $Z_{\mu_{M\times(-\cO)}}$ defined in \eqref{def func d}, applied to the diagonal action by $K$ on $M \times (-\cO)$, and the moment map $\mu_{M \times (-\cO)}$.
The following important property of the function $d_{\cO}$ was proved in Theorem 4.20 in \cite{Paradan14}.
\begin{proposition}\label{prop dO nonneg}
The function $d_{\cO}$ is nonnegative.
\end{proposition}

\begin{remark}
If one is interested in the $K$-invariant part of $\QSpinc_K(M, \mu)$, one can replace $M$ by a manifold on which the function $d$ is nonnegative. Indeed, the fact that that $\pi_{K\cdot \rhoK}$ is the trivial representation implies that
\[
\begin{split}
\QSpinc_K(M, \mu)^K &= \bigl( \QSpinc_K(M, \mu) \otimes \QSpinc(-K\cdot \rhoK)\bigr)^K \\
&= \QSpinc_K\bigl(M\times (- K\cdot \rhoK), \mu_{M \times (- K\cdot \rhoK)} \bigr)^K. 
\end{split}
\]
So one can work with the manifold $M\times (- K\cdot \rhoK)$. On that manifold, the fact that
 the orbit $K\cdot \rhoK$ is regular and admissible implies that  the function $d_{K\cdot \rhoK}$ is nonnegative, by Proposition \ref{prop dO nonneg}.
\end{remark}

\subsection{Localising multiplicities} \label{sec loc mult}

As before, consider a regular, admissible coadjoint orbit $\cO$ of $K$.  Let $\pi_{\cO} := \QSpinc_K(\cO)$ be the corresponding irreducible representation of $K$. It is noted  in Proposition 3.9 in \cite{Paradan15} that every irreducible representation can be realised in this way, for precisely one regular, admissible orbit $\cO$.  In addition, $-\cO$ is also regular and admissible, and $\pi_{-\cO}$ is the dual representation to $\pi_{\cO}$. 

Let $m_{\cO} \in \Z$ be the multiplicity of $\pi_{\cO}$ in $\QSpinc_K(M, \mu)$. By applying Theorem \ref{thm mult}, with $N = -\cO$, we obtain an extension of the shifting trick:
\begin{equation} \label{eq m O 1}
m_{\cO} = \left(\QSpinc_K(M, \mu) \otimes \QSpinc_K(-\cO) \right)^K = \QSpinc_K\bigl(M \times (-\cO), \mu_{M \times (-\cO)}  \bigr)^K.
\end{equation}
This generalises the expression for $m_{\cO}$ at the start of Section 4.5.3 in \cite{Paradan14}, and is the basis of the proof of Theorem \ref{thm [Q,R]=0}. It also allows us to prove Theorem \ref{thm kM h}. Let $d_{\cO}$ be the function in Proposition \ref{prop dO nonneg}. 

\medskip \noindent \emph{Proof of Theorem \ref{thm kM h}.}
Suppose that $Q^{\Spinc}_K(M, \mu)\not=0$. Then there is a regular, admissible coadjoint orbit $\cO$ such that $m_{\cO}\not=0$. Because of \eqref{eq m O 1}, this means that
\[
\QSpinc_K\bigl(M \times (-\cO), \mu_{M \times (-\cO)}  \bigr)^K \not = 0.
\]
Proposition \ref{prop vanishing d} then implies that the function $d_{\cO}$ is not strictly positive. By the first point in the second part of Theorem 4.20 in \cite{Paradan14}, this implies that there is a class $(\kh)\in \cH_{\kk}$ such that $([\kk^M, \kk^M]) = ([\kh, \kh])$.
\hfill $\square$

Now consider the compact set
\[
Z_{\mu_{M \times (-\cO)}}^{=0} := d_{\cO}^{-1}(0).
\]
By Lemma \ref{lem 4.16}, there are disjoint, $K$-invariant open subsets $U^{<0}, U^{=0}, U^{>0} \subset M\times (-\cO)$ such that 
$Z_{\mu_{M \times (-\cO)}}^{=0}  \subset U^{=0}$, and
$d_{\cO}$ is negative on $U^{<0}  \cap Z_{\mu_{M \times (-\cO)}}$ and positive on $U^{>0} \cap Z_{\mu_{M \times (-\cO)}}$. 
By Proposition \ref{prop excision}, we have
  \begin{multline*}
 \QSpinc_K\bigl(M\times (-\cO), \mu^{M\times(-\cO)} \bigr)^K = \\
 \QSpinc_K \bigl(U^{<0}, \mu^{M\times(-\cO)}|_{U^{<0}} \bigr)^K + 
 \QSpinc_K \bigl(U^{=0}, \mu^{M\times(-\cO)}|_{U^{=0}} \bigr)^K +  \QSpinc_K\bigl(U^{>0}, \mu^{M\times(-\cO)}|_{U^{>0}} \bigr)^K.
 \end{multline*}
Now the first term on the right hand side vanishes by Proposition \ref{prop dO nonneg}, while the last term vanishes by Proposition \ref{prop vanishing d}. Therefore, we obtain a localised expression for the multiplicity $m_{\cO}$.
\begin{corollary} \label{cor mO loc}
For any $K$-invariant,  open neighbourhood $U^{=0}$ of $Z_{\mu_{M \times (-\cO)}}^{=0} $ such that $U^{=0} \cap Z_{\mu_{M \times (-\cO)}} =Z_{\mu_{M \times (-\cO)}}^{=0} $, the multiplicity $m_{\cO}$ of $\pi_{\cO}$ in $\QSpinc_K(M, \mu)$ equals
\[
m_{\cO} = \QSpinc_K(U^{=0}, \mu^{M\times(-\cO)}|_{U^{=0}})^K.
\]
\end{corollary}

\subsection{Decomposing multiplicities} \label{sec decomp mult}

The expression for the multiplicity $m_{\cO}$ in Corollary \ref{cor mO loc}
 is the same as the expression just below the first display at the start of Section 4.5.3 in \cite{Paradan14}. In addition, the set $Z_{\mu_{M \times (-\cO)}}^{=0}$ is compact, so that the set $U^{=0}$ may be chosen to be relatively compact. So from here on, the situation is exactly the same as in \cite{Paradan14}. Therefore, the arguments needed to deduce Theorem \ref{thm [Q,R]=0} from Corollary \ref{cor mO loc} are the same as those used in \cite{Paradan14} to deduce Theorem 5.9 in that paper from the localised expression for $m_{\cO}$ at the start of Section 4.5.3. We finish the proof of Theorem \ref{thm [Q,R]=0} by summarising these arguments and how to apply them in our setting.

Since Proposition 4.25 in \cite{Paradan14} is stated and proved without assuming the manifold $M$ to be compact, the decomposition (4.30) in \cite{Paradan14} still holds in our setting:
\begin{equation} \label{eq m O P}
m_{\cO} = \sum_{\cP} m_{\cO}^{\cP},
\end{equation}
with $m_{\cO}^{\cP}$ as defined below (4.30) in \cite{Paradan14}.
Here the sum runs over all admissible coadjoint orbits $\cP$ of $G$ such that $\QSpinc_K(\cP) = \pi_{\cO}$ and the stabilisers of points in $\cP$ are conjugate to $\kk_{\xi}$, where $\xi \in \kk^*$ is an element such that
\[
[\kk^M, \kk^M] = [\kk_{\xi}, \kk_{\xi}].
\]
The integers $m_{\cO}^{\cP}$ can be computed in terms of actions by tori, by Theorem 4.29 in \cite{Paradan14}. 

This theorem is based on Propositions 4.15 and 4.28 in \cite{Paradan14}. Proposition 4.28 in \cite{Paradan14}, and its proof, remain true without changes for noncompact manifolds. In Proposition 4.15 in \cite{Paradan14}, one considers a {compact} component of the set or zeroes of the vector field induced by a moment map. Since the set $Z^{=0}_{\mu_{M \times (-\cO)}}$ is compact, this proposition still applies in our setting.
Therefore, Proposition 4.15 in \cite{Paradan14} can be used to show that the expression for $m_{\cO}^{\cP}$ above Proposition 4.28 in \cite{Paradan14} is true in the proper moment map case. It then follows that Theorem 4.29 in \cite{Paradan14} generalises to this more general case, because the remainder of its proof is a local computation. 

This finally  allows one to prove Theorem \ref{thm [Q,R]=0}. In Sections 5.1 and 5.2 of \cite{Paradan14}, possibly noncompact $\Spinc$-manifolds with proper moment maps are considered, so the results there apply in our setting. Using Proposition 5.8 in \cite{Paradan14}, and the definition of quantisation of reduced spaces on page 53 of \cite{Paradan14}, one concludes that Theorem 5.9 in \cite{Paradan14} generalises to the proper moment map case, which is to say that 
the expression for $m_{\lambda}$ in 
Theorem \ref{thm [Q,R]=0} is true.

\bibliographystyle{plain}
\bibliography{mybib}

\end{document}